\documentclass[twoside,11pt]{article}

\usepackage{jmlr2e}

\usepackage{graphicx,amssymb,amsmath}

\usepackage{enumerate,psfrag}
\usepackage{hyperref}
\DeclareGraphicsExtensions{.eps,.png,.jpg}
\usepackage[sort]{cite}

\usepackage{tcolorbox}
\hypersetup{breaklinks=true,colorlinks=true,linkcolor=blue,citecolor=blue}

\usepackage[noabbrev,capitalise,nameinlink]{cleveref}

\usepackage{algorithm}
\usepackage{algpseudocode}
\allowdisplaybreaks

\newtheorem{assumption}{Assumption}

\ShortHeadings{Stochastic Primal-Dual Q-Learning}{Lee and He}
\firstpageno{1}

\begin{document}

\title{Stochastic Primal-Dual Q-Learning}

\author{\name Narim Jeong \email nrjeong@kaist.ac.kr\\
       \addr Department of Electrical Engineering\\
       Korea Advanced Institute of Science and Technology (KAIST)\\
       Daejeon, 34141, South Korea
       \AND
\name Donghwan Lee \email donghwan@kaist.ac.kr\\
       \addr Department of Electrical Engineering\\
       Korea Advanced Institute of Science and Technology (KAIST)\\
       Daejeon, 34141, South Korea
       \AND
       \name Niao He \email niao.he@inf.ethz.ch \\
       \addr Department of Computer Science\\
       ETH Zurich\\
       Zurich, 8092, Switzerland}

\editor{not determined}

\maketitle

\begin{abstract}
In this work, we present a new model-free and off-policy reinforcement learning (RL) algorithm, that is capable of finding a near-optimal policy with state-action observations from arbitrary behavior policies. Our algorithm,  called the stochastic primal-dual Q-learning (SPD Q-learning),
hinges upon a new linear programming formulation and a dual perspective of the standard Q-learning. In contrast to previous primal-dual RL algorithms, the SPD Q-learning includes a Q-function estimation step, thus allowing to recover an approximate policy from the primal solution as well as the dual solution. We prove a first-of-its-kind result that the SPD Q-learning guarantees a certain convergence rate, even when the state-action distribution is time-varying but sub-linearly converges to a stationary distribution. Numerical experiments are provided to demonstrate the off-policy learning abilities of the proposed algorithm in comparison to the standard Q-learning.
\end{abstract}

\begin{keywords}
 Reinforcement learning (RL), Saddle point problem, Markov decision process (MDP), Q-learning
\end{keywords}

\section{Introduction}

The problem of learning a map from world observations to actions, called a policy, lies at the core of many sequential decision problems, such as robotics~\citep{chen2017socially}, artificial
intelligence~\citep{mnih2015human}, finance~\citep{longstaff2001valuing}, and economics~\citep{tesauro2002pricing}. The development of policies is often very challenging in many
real-world applications as finding accurate world models is difficult under complex interactions between the decision maker and environment. Reinforcement learning (RL)~\citep{sutton1998reinforcement,bertsekas1996neuro,puterman2014markov} is a subfield of machine learning which addresses the problem of how an autonomous agent (decision maker) can learn an optimal policy to maximize long-term cumulative rewards, while interacting with unknown environment.

Many classical RL algorithms, e.g., temporal difference methods~\citep{sutton1988learning}, Q-learning~\citep{watkins1992q}, SARSA~\citep{rummery1994line}, are based on the sample-based stochastic dynamic programming to solve the Bellman equation, taking advantage of its contraction mapping property to guarantee their convergence. Comprehensive reviews of the dynamic programming and RL approaches can be found in the books~\citet{sutton1998reinforcement,bertsekas1996neuro,puterman2014markov}. Recently, there has been a growing interest in integrating Bellman equations into optimization frameworks to design provably efficient RL algorithms, by leveraging the existing fruitful optimization algorithms and theories. See, e.g.,~\citet{baird1995residual, sutton2009fast,mahadevan2014proximal,dai2017learning} for policy evaluation and \citet{wang2016online,chen2016stochastic,dai2018boosting,dai2018sbeed} for policy design. In particular, \citet{chen2016stochastic} considers a linear programming form of the Bellman equation and introduces a stochastic primal-dual (SPD) algorithm to solve the min-max problem of the associated Lagrangian function, assuming samples from a uniform state-action distribution. The primal-dual optimization perspective is further employed in~\citet{dai2018boosting} with nonlinear function approximations to solve Markov decision problems with continuous state-actions spaces. Besides the direct advantage of theoretical guarantees, such optimization frameworks are also very favorable and extensible when dealing with constraints, sparsity regularizations~\citep{mahadevan2012sparse}, and distributed scenarios~\citep{kar2013cal,macua2015distributed,zhang2018fully,lee2018primal}.

{\em Statement of Contributions:} Although substantial advances have been made recently in this direction, to the authors' knowledge, finding a reliable suboptimal policy using samples from real-world trajectories remains largely unexplored, leaving a huge gap between theory and practice. Inspired by the above discussions, this paper centers at filling in this gap by proposing a new linear programming (LP) formulation of the standard Q-learning~\citep{watkins1992q}, known to be one of the most popular RL algorithms for policy design, in order to leverage its powerful model-free and off-policy learning abilities to solve Markov decision making problems. The main contributions are summarized below.
\begin{enumerate}
\item We develop a novel stochastic primal-dual Q-learning (SPD Q-learning) algorithm to solve the corresponding Lagrangian of the LP, that uses only samples of real-world trajectories without any importance sampling steps, as usually required in off-policy RL algorithms~\citep{precup2001off}. The proposed algorithm includes a Q-function estimation step, and allows recovering an optimal policy using the primal solutions as well as the dual solutions.

\item  Moreover, the SPD Q-learning is the first RL which guarantees the convergence with a certain convergence rate even when the underlying distribution of the state-action observations is time-varying but sub-linearly converges to a stationary distribution. This result applies to important cases, such as  when the state distribution under a fixed behavior policy is time-varying, or when the behavior policy itself is time-varying.

\item We provide a detailed convergence and sample complexity analysis for the SPD Q-learning algorithm. In particular, we prove that with the number of iterations/samples at least ${\cal O}\left(\frac{|{\cal S}|^4 |{\cal A}|^4}{\zeta^4(1-\alpha)^4}\frac{1}{\varepsilon^2}\ln \left(\frac{1}{\delta}\right)\right)$, the algorithm generates a candidate solution with duality gap less than or equal to $\varepsilon$ with probability $1-\delta$, where $|{\cal S}|$ is the number of the states, $|{\cal A}|$ is the number of the actions, $\alpha \in (0,1)$ is the discount factor, and $\zeta \in (0,1)$ is a constant related to the state-action distribution. This result also leads to the conclusion that with the number of iterations at least ${\cal O}\left(\frac{|{\cal S}|^6 |{\cal A}|^4}{\zeta^4(1-\alpha)^6}\frac{1}{\varepsilon^2}\ln \left(\frac{1}{\delta}\right)\right)$, an $\varepsilon$-suboptimal policy can be recovered from the algorithm with probability at least $1-\delta$, where the policy is $\varepsilon$-suboptimal in the sense that the distance between the optimal value function and the value function corresponding to the obtained policy is less than or equal to $\varepsilon$.

\item To demonstrate the validity of the proposed algorithm, we provide simulation results for simple Markov decision making problems. Through the simulations, we observe that the suboptimal policy recovered from the primal solution converges faster than the suboptimal policy from the dual solution, and this is a potential advantage of the proposed algorithm.

\end{enumerate}
We expect that this fundamental framework will be useful to advance many subfields of RL, such as the distributed RL~\citep{kar2013cal,macua2015distributed,zhang2018fully,lee2018primal}, Q-learning with function approximations~\citep{sutton2009convergent,sutton2009fast}, sparsity promoted RL~\citep{mahadevan2012sparse}, safe RL~\citep{garcia2015comprehensive}, and the inverse RL~\citep{ng2000algorithms}.

The remainder of the paper is organized as follows. \cref{section:preliminaries} contains preliminaries including notations, definitions, problem formulations, standard LP formulation of the dynamic programming, and its solution analysis. \cref{section:SPD-Q-learning} proposes a new LP formulation of the dynamic programming tailored to the proposed Q-learning algorithm, its solution analysis, and the main SPD Q-learning algorithm. The corresponding convergence results of the SPD Q-learning algorithm are summarized in~\cref{section:main-result}, and detailed proofs are included in~\cref{section:convergence}. Simulation results are given in~\cref{section:simulations}, and finally, we provide conclusions in~\cref{section:conclusion}.

{\bf Notation:} The following notation is adopted: ${\mathbb R}^n $ denotes the
$n$-dimensional Euclidean space; ${\mathbb R}^{n \times m}$
denotes the set of all $n \times m$ real matrices; ${\mathbb R}_+^n$
and ${\mathbb R}_{++}^n$ denote the sets of vectors with nonnegative and positive
real elements, respectively, $A^T$ denotes the transpose of matrix
$A$; $I_n$ denotes the $n \times n$ identity matrix; $I$ denotes
the identity matrix with  appropriate dimension; $\|\cdot
\|_1,\|\cdot \|_2,\|\cdot \|_\infty$ denote the standard matrix
1-norm, Euclidean norm, and $\infty$-norm, respectively; $|{\cal
S}|$ denotes the cardinality of the set for any finite set ${\cal
S}$; ${\mathbb E}[\cdot]$ denotes the expectation operator;
${\mathbb P}[\cdot]$ denotes the probability of an event; $x(i)$
is the $i$-th element for any vector $x$; $P(i,j)$ indicates the
element in $i$-th row and $j$-th column for any matrix $P$; if
${\bf z}$ is a discrete random variable which has $n$ values and
$\mu \in {\mathbb R}^n$ is a stochastic vector, then ${\bf z} \sim
\mu$ stands for ${\mathbb P}[{\bf z} = i] = \mu(i)$ for all $i
\in \{1,\ldots,n \}$; ${\bf 1}_n \in {\mathbb R}^n$ denotes an
$n$-dimensional vector with all entries equal to one; for a convex
closed set $\cal S$, $\Pi_{\cal S}(x)$ is the projection of $x$
onto the set $\cal S$, i.e., $\Pi_{\cal S}(x):={\rm argmin}_{y\in {\cal
S}} \|x-y\|_2$; $\Delta_n$ with a positive integer $n$ is the unit
simplex defined as $\Delta_n:= \{(\alpha_1,\ldots,\alpha_n):\alpha_1 + \cdots + \alpha_n = 1,\,\alpha_i \ge 0,\forall i \in \{1,\ldots,n\}\}$; $e_j,j\in \{ 1,2,\ldots,n\}$, is
the $j$-th basis vector (all components are $0$ except for the
$j$-th component which is $1$) of appropriate dimensions.

\section{Preliminaries}\label{section:preliminaries}

\subsection{Problem formulation}
In this paper, we consider the infinite-horizon discounted Markov decision problem (MDP), where the agent tries to take actions to maximize
cumulative discounted rewards over infinite time horizons. In
particular, an instance of the discounted MDP can be represented by the
tuple $({\cal S},{\cal A},{\cal P},{\cal R},\alpha)$, where ${\cal
S}:=\{ 1,2,\ldots ,|{\cal S}|\}$ is a discrete state-space of size
$|{\cal S}|$, ${\cal A}:= \{1,2,\ldots,|{\cal A}|\}$ is a discrete
action-space of size $|{\cal A}|$, $\alpha \in [0,1)$ is the
discount factor, ${\cal P}$ defines a collection of state-to-state
transition probabilities, ${\cal P}: = \{P_a\in {\mathbb
R}^{|{\cal S}| \times |{\cal S}|},a \in {\cal A}\}$, where $P_a
(s,s')$ is the state transition probability from the current state
$s\in {\cal S}$ to the next state $s' \in {\cal S}$ under action
$a \in {\cal A}$, ${\cal R}: = \{ \hat r_{ss'a}  \in [0,\sigma ],a
\in {\cal A},s,s' \in {\cal S}\}$ is a collection of reward random
variables, where $\sigma>0$ is a real number and $\hat r_{ss'a}$
is the random reward when the current state, next state, and action is
$s,s',a$, respectively, with its expectation ${\mathbb E}[\hat
r_{ss'a}]= r_{ss'a}$. Without loss of generality and for simplicity, we assume $\sigma \geq 1$ throughout the paper. Let $\pi :{\cal S} \to {\cal A}$
be a deterministic policy that maps a state $s \in {\cal S}$ to an
action $\pi(s) \in {\cal A}$. With abuse of notation, the deterministic policy is interchangeably represented by the stochastic vector $\pi_s \in \Delta_{|{\cal A}|}$ such that $\pi_s = e_{\pi(s)} \in \Delta_{|{\cal A}|}$, where $e_i$ is the $i$-th basis vector in ${\mathbb R}^|{\cal A}|$. Hereafter, the dimension of $e_i$ is not specified if it is clear from the context. We denote the state-to-state
transition probability matrix under the deterministic policy $\pi$
by $P_\pi$, where $P_\pi (s,s'):=P_{\pi(s)} (s,s')$ for $s,s'
\in {\cal S}$. The infinite-horizon discounted cost under policy
$\pi$ is defined as
\begin{align}
&V^{\pi}(s) = {\mathbb E}\left[ \left. \sum_{k = 0}^\infty
{\alpha^k \hat r_{s_k s_{k + 1}\pi(s_k)}} \right|s_0 = s
\right],\quad s\in {\cal S},\label{eq:cost}
\end{align}
where $(s_0,s_1,\ldots)$ is a state trajectory generated by the
Markov chain under policy $\pi$. The discounted Markov decision
making problem is to find a deterministic optimal policy, $\pi^*: {\cal S} \to {\cal A}$,
such that the infinite-horizon discounted cost $V^{\pi}$ is
maximized, i.e.,
\begin{align*}
&\pi^* :={\rm argmax}_{\pi :{\cal S} \to {\cal A}} {\mathbb E}\left[
\sum_{k = 0}^\infty  {\alpha^k \hat r_{s_k s_{k+1} \pi (s_k )}}\right].
\end{align*}
Note that the optimal policy is always deterministic~\citep{puterman2014markov}. The main goal is to solve the decision problem by finding the optimal policy.

\subsection{LP formulation of dynamic programming}

In this subsection, we briefly review a linear programming (LP)
formulation of the dynamic programming problem
from~\citet{puterman2014markov,wang2016online,chen2016stochastic}.
Associated with~\eqref{eq:cost}, the optimal cost vector, $V^* \in
{\mathbb R}^{|{\cal S}|}$, is defined as
\begin{align*}
V^*(s)&:= V^{\pi^*}(s)={\mathbb E}\left[\left. \sum_{k=0}^\infty
{\alpha^k \hat r_{s_k s_{k+1} \pi^*(s_k)}} \right|s_0 = s \right]=\max_{\pi:{\cal S} \to {\cal A}} {\mathbb E}\left[\left. \sum_{k=0}^\infty {\alpha^k \hat r_{s_k s_{k + 1} \pi(s_k)}}
\right|s_0=s\right]
\end{align*}
for $s \in {\cal S}$. We will consider a general stochastic policy
denoted by $\theta_s \in \Delta_{|{\cal A}|},s \in {\cal S}$, where
$\theta_s(a),s \in {\cal S},a \in {\cal A}$, is the probability of
taking action $a \in {\cal A}$ when the current state is $s \in
{\cal S}$. The state-to-state transition probability matrix under
the stochastic policy $\theta$ is denoted by $P_\theta$, where
\begin{align}
&P_\theta=\sum_{a \in {\cal A}}{ \begin{bmatrix}
   \theta_1(a) & &  \\
    & \ddots  & \\
    & & \theta_{|{\cal S}|}(a)  \\
\end{bmatrix} P_a}.\label{eq:def:P-mu}
\end{align}
Note that if $\theta_s\in\Delta_{|A|},s \in {\cal S}$, is a standard
basis vector, then it is reduced to the deterministic case. In
addition, define the expected reward $R_a(s)$ conditioned on the
current action $a$ and state $s$, i.e., $R_a(s): = \sum_{s'\in {\cal S}} {P_a (s,s')r_{ss'a}}$, and the corresponding vectors
\begin{align*}
&R_a\in {\mathbb R}^{|{\cal S}|},\quad R:= \begin{bmatrix}
R_1\\
\vdots\\
R_{|A|}\\
\end{bmatrix}\in {\mathbb R}^{|{\cal S}||{\cal A}|}.
\end{align*}
Similarly, for any stochastic policy $\mu_s\in \Delta_{|{\cal
A}|},s \in {\cal S}$, $R_\mu(s),s \in {\cal S}$, is defined as
\begin{align}
&R_\mu(s): = \sum_{a\in {\cal A}} {\sum_{s'\in {\cal S}} {\mu_s(a)
P_a (s,s')r_{ss'a}} }= \sum_{a \in {\cal A}} {\mu_s(a)
R_a(s)}.\label{eq:def:R-mu}
\end{align}

It is well-known~\citep{bertsekas1996neuro,puterman2014markov,chen2016stochastic}
that the optimal cost vector, $V^*\in {\mathbb R}^{|{\cal S}|}$,
can be obtained by solving the linear programming problem (LP)
\begin{align*}
&\min_{V\in {\mathbb R}^{|{\cal S}|}} \quad \eta^T V\quad {\rm s.t.}\quad \alpha P_a V + R_a \le
V,\quad a \in {\cal A},
\end{align*}
where $\eta\in {\mathbb R}^{|{\cal S}|}$ is any vector with
positive elements and `$\le$' is the element-wise inequality. Introducing the notation
\begin{align*}
&P:= \begin{bmatrix}
P_1\\
\vdots\\
P_{|{\cal A}|}\\
\end{bmatrix} \in {\mathbb R}^{|{\cal S}||{\cal A}| \times |{\cal
S}|},
\end{align*}
the LP is compactly written by
\begin{align}
&p^*:=\min_{V\in {\mathbb R}^{|{\cal S}|}} \quad \eta^T V\quad {\rm s.t.}\quad R + \alpha PV \le
({\bf 1}_{|{\cal A}|}\otimes I_{|{\cal
S}|})V.\label{eq:DP-LP-form}
\end{align}
We will call this LP as the {\em primal problem}. Before the
development of the main results, some preliminary results are
introduced. First, the optimal solution of the
LP~\eqref{eq:DP-LP-form} is unique.
\begin{lemma}[{\citet[Theorem~1]{chen2016stochastic}}]\label{lemma:uniqueness-of-LP-DP}
The LP~\eqref{eq:DP-LP-form} has the unique solution~$V^*=
(I_{|{\cal S}|}-\alpha P_{\pi^*})^{-1} R_{\pi^*}$.
\end{lemma}

It is meaningful to consider the dual LP of~\eqref{eq:DP-LP-form}
because its dual solution is known to be closely related to the
optimal policy
$\pi^*$~\citep{bertsekas1996neuro,puterman2014markov,chen2016stochastic}.
In particular, consider the Lagrangian function
\begin{align*}
&L(V,\lambda)= \eta^T V +\lambda^T (R+ \alpha PV - {\bf 1}_{|{\cal
S}|}\otimes V),
\end{align*}
where $\lambda:= \begin{bmatrix}\lambda_1^T & \cdots &
\lambda_{|{\cal S}||{\cal A}|}^T\end{bmatrix}^T \in {\mathbb
R}^{|{\cal S}||{\cal A}|}$ is the Lagrangian multiplier. Using the
standard results in convex optimization theories, LPs satisfy the
Slater's condition~\citep[Chapter~5]{Boyd2004}, and by the strong duality, the min-max problem
satisfies
\begin{align}\label{eq:saddlepont}
&\min_{V \in {\mathbb R}^{|{\cal S}|}} \max_{\lambda \geq 0} L(V,\lambda)= \max_{\lambda \geq 0} \min_{V \in {\mathbb R}^{|{\cal S}|}} L(V,\lambda).
\end{align}
According to~\citet[Prop.~2.6.1,pp.~132]{bertsekas2003convex}, one concludes that there exists a saddle point $(V^*,\lambda^*)$ satisfying $L(V^*,\lambda)\le L(V^*,\lambda^*)\le L(V,\lambda^*),\forall (V,\lambda) \in {\cal X} \times {\cal Y}$, with ${\cal X}= {\mathbb R}^{|{\cal S}|}$ and ${\cal Y} = {\mathbb R}_+^{|{\cal S}||{\cal A}|}$. In addition, $V^*$ is an optimal
solution of the primal problem~\eqref{eq:DP-LP-form} and $\lambda^*$ an optimal solution of the {\em dual problem}
\begin{align}
&d^*=\max_{\lambda\ge 0}\quad\lambda^T R \quad {\rm s.t.}\quad
\eta + \alpha P^T \lambda = ({\bf 1}_{|{\cal A}|} \otimes I)^T
\lambda.\label{eq:dual-LP}
\end{align}

Similarly to the primal
LP~\eqref{eq:DP-LP-form}, the dual solution is unique, and its
expression can be obtained as follows.
\begin{lemma}[{\citet[Theorem~1]{chen2016stochastic}}]\label{lemma:uniqueness-of-dual-LP-DP}
The dual LP~\eqref{eq:dual-LP} has the unique solution
$\lambda^*:=\begin{bmatrix}
   (\lambda_1^*)^T & \cdots & (\lambda_{|{\cal A}|}^*)^T\\
\end{bmatrix}^T \in {\mathbb R}^{|{\cal S}||{\cal A}|}$ with $\lambda_a^*:=\begin{bmatrix}
   \lambda_a^*(1) & \cdots & \lambda_a^*(|{\cal S}|)\\
\end{bmatrix}^T\in {\mathbb R}^{|{\cal S}|}$ satisfying
\begin{align*}
&\begin{bmatrix}
\lambda_{\pi^*(1)}^*(1)  \\
\vdots\\
\lambda_{\pi^*(|{\cal S}|)}^*(|{\cal S}|)\\
\end{bmatrix} = (I - \alpha(P_{\pi^*})^T)^{-1} \eta,\quad\lambda_a^*(s)= 0\quad {\rm if}\quad a \ne \pi^*(s),\quad s \in {\cal
S}.
\end{align*}
\end{lemma}

Once the dual optimal solution is obtained, then the optimal policy can be recovered by
\begin{align*}
&\pi^*(s): = \begin{bmatrix}
   \frac{\lambda_{1}^*(s)}{\sum_{a'\in {\cal A}}{\lambda_{a'}^*(s)}} &\cdots  & \frac{\lambda_{|{\cal A}|}^*(s)}{\sum_{a'\in {\cal A}}{\lambda_{a'}^*(s)}}\\
\end{bmatrix}^T\in\Delta_{|{\cal A}|}.
\end{align*}
It is known that the optimal policy is always deterministic~\citep{puterman2014markov}. In summary, the Markov decision problem can be solved by finding the optimal solution of the dual LP~\eqref{eq:dual-LP},
while the optimal cost $V^*$ can be found by solving the primal
LP~\eqref{eq:DP-LP-form}.

Based on~\cref{lemma:uniqueness-of-dual-LP-DP}, we obtain the
following corollary.
\begin{corollary}\label{corollary:uniqueness-of-dual-LP-DP}
The dual LP~\eqref{eq:dual-LP} solution $\lambda^*:=
\begin{bmatrix}
(\lambda_1^*)^T & \cdots & (\lambda_{|{\cal A}|}^*)^T \\
\end{bmatrix}^T\in {\mathbb R}^{|{\cal S}||{\cal A}|}$ satisfies
\begin{align*}
&\eta= (I - \alpha (P_{\pi^*})^T )\sum_{a \in {\cal
A}}{\lambda_a^*}.
\end{align*}
\end{corollary}
Based on~\cref{lemma:uniqueness-of-LP-DP} and~\cref{lemma:uniqueness-of-dual-LP-DP}, bounds on optimal
primal and dual solutions can be obtained, and those bounds are used
in the next section in the algorithm development and its analysis.
\begin{lemma}\label{lemma:bound-lemma}
Let $(V^*,\lambda^*)$ be the unique optimal primal-dual pair
solving~\eqref{eq:DP-LP-form} and~\eqref{eq:dual-LP}. Then,
\begin{align*}
&\|V^*\|_\infty \le\frac{\sigma}{1-\alpha},\quad \|V^*\|_2  \le
\frac{\sqrt{|{\cal S}|}\sigma}{1-\alpha},\quad \|\lambda^*\|_2 \le
\| \lambda^* \|_1\le \frac{\|\eta\|_1}{1-\alpha},\quad \sum_{a\in
{\cal A}}{\lambda_a^*}\ge\eta,\\
&\|\lambda^* \|_\infty\le\frac{\|\eta\|_1}{1-\alpha},
\end{align*}
where $\sigma>0$ is an upper bound on the random reward, $\hat r_{ss'a} \in [0,\sigma]$, $s,s'\in {\cal S},a\in {\cal A}$.
\end{lemma}
\begin{proof}
Proofs of the first four results can be found
in~\citet[Lemma~1]{chen2016stochastic}. The last result is obtained
from the inequality $\|\lambda^* \|_\infty \le \|\lambda^*\|_2$
for any $\lambda^*$ and the third result.
\end{proof}

\subsection{Saddle point problem}
In this subsection, we briefly introduce the concept of the saddle point problem.
\begin{definition}[Saddle
point~{\citep[Def.~2.1.6, pp.~131]{bertsekas2003convex}}]\label{def:saddle-point}
Consider the map $L:{\cal X} \times {\cal Y} \to {\mathbb R}$,
where ${\cal X}$ and ${\cal Y}$ are convex sets. A pair $(x^*,y^*)$ that satisfies
\begin{align*}
&L(x^*,y)\le L(x^*,y^*)\le L(x,y^*),\quad \forall (x,y) \in {\cal
X} \times {\cal Y}
\end{align*}
is called, if exists, a saddle point of $L$. The saddle
point problem is defined as the problem of finding a saddle point
$(x^*,y^*)$.
\end{definition}
Note that $(x^*,y^*)$ is a saddle point if and only if $x^* \in {\cal X}$, $y^* \in {\cal Y}$, and $\sup_{y\in {\cal Y}} L(x^*,y)=L(x^*,y^*)=\inf_{x\in {\cal X}}L(x,y^*)$. The following proposition establishes a relation between the saddle point and optimization problems.
\begin{proposition}[{\citet[Prop.~2.6.1, pp.~132]{bertsekas2003convex}}]\label{proposition:saddle-point-relation}
The point $(x^*,y^*)$ is a saddle point of $L$ if and only if (a) $x^* \in {\cal X}$, $y^* \in {\cal Y}$, and $\sup_{y\in {\cal Y}} L(x^*,y)=L(x^*,y^*)=\inf_{x\in {\cal X}}L(x,y^*)$, (b) $x^*$ is an optimal solution of the {\rm primal problem} $\min_{x\in {\cal X}}\{\overline{L}(x):=\max_{y \in {\cal Y}}L(x,y)\}$, and (c) $y^*$ is an optimal solution of the {\rm dual problem} $\max_{y \in {\cal Y}}\{\underline{L}(y):=\min_{x\in {\cal X}}L(x,y)\}$.
\end{proposition}
Lastly, we formally define the saddle point problem.
\begin{definition}[Saddle point problem]
Consider the map $L:{\cal X} \times {\cal Y} \to {\mathbb R}$,
where ${\cal X}$ and ${\cal Y}$ are convex sets. Assume that the saddle point $(x^*,y^*)$ of $L$ exists. Then, the saddle point problem is defined as the problem of finding saddle points
$(x^*,y^*)$ which satisfy the primal and dual optimizations
\begin{align*}
&\max_{y \in {\cal Y}}\min_{x \in {\cal X}} L(x,y)=\min_{x\in {\cal X}}\max_{y\in {\cal Y}} L(x,y).
\end{align*}
\end{definition}

\subsection{Stochastic primal-dual RL}
Recently, a stochastic primal-dual algorithm (SPD-RL) was proposed in~\citet{wang2016online,chen2016stochastic} to solve the convex-concave saddle point problem in (\ref{eq:saddlepont}), which updates the primal and dual solutions simultaneously using noisy estimates of partial derivatives of the Lagrangian
function obtained from samples of state-action transitions. Particularly, the SPD-RL algorithm in~\citet{wang2016online} uses the uniform state-action distribution to sample the current state and action. From this observation, a natural question arises: can we develop an SPD-RL algorithm under stationary state-action distributions induced from behavior policies? This question is important in terms of applicability of the SPD-RL to real-world learning tasks where samples are only available from the state-action trajectories. One possible approach is to solve the saddle point problem corresponding to the modified LP
\begin{align}
&\min_V \quad \eta^T V\quad {\rm s.t.}\quad \alpha M_a P_a V + M_a R_a \le
M_a V,\quad a\in {\cal A},\label{intro:LP1b}
\end{align}
where $M_a$ is a positive diagonal matrix whose diagonal elements are the stationary state-action distribution under a certain behavior policy with the fixed action $a \in {\cal A}$. While this approach successfully estimates the optimal value function $V^*$, it fails to recover the optimal policy $\pi^*$  because the dual optimal solution of~\eqref{intro:LP1b}, $\{M_a^{-1} \lambda_a^*\}_{a\in {\cal A}}$, is different from that of~\eqref{eq:DP-LP-form}. This implies that to obtain the exact dual optimal solution, one needs the knowledge of the state-action distribution, $M_a,a\in {\cal A}$, which is not directly available without additional sampling and estimation steps.

\section{Stochastic Primal-Dual Q-Learning Algorithm}\label{section:SPD-Q-learning}

In this section, we introduce an SPD Q-learning algorithm to overcome the challenges described in the last subsection by integrating the primal-dual algorithm with Q-learning. The main advantage of the Q-learning lies in the fact that instead of the value function, it estimates the value function, $Q_a,a\in {\cal A}$, corresponding to the state-action pair, called the Q-function, and the optimal policy can be directly recovered from the optimal Q-function, $Q_a^*,a\in {\cal A}$, without the model information, i.e., $\pi^*(s)={\rm argmax}_a Q_a^*(s),s \in {\cal S}$. For any given deterministic
policy $\pi$, the action value function or
Q-function~\citep{bertsekas1996neuro} is defined as
\begin{align*}
Q_a^\pi(s):=& {\mathbb E}\left[ \left. \sum_{k=0}^\infty {\alpha^k
\hat r_{s_k s_{k + 1} a_k}} \right|s_0=s,a_0= a,a_k=\pi(s_k),k \ge 1\right],
\end{align*}
and the optimal Q-function is $Q_a^*(s):=Q_a^{\pi^*}(s)$. Consider the
corresponding vector
\begin{align*}
&Q_a^\pi:= \begin{bmatrix}
Q_a^\pi(1) & \cdots & Q_a^\pi(|{\cal S}|) \\
\end{bmatrix}^T\in {\mathbb R}^{|{\cal S}|}.
\end{align*}
Using the definition of $R_a \in {\mathbb R}^{|{\cal
S}|}$ and the Q-function, one easily proves the relation between $Q_a^*$ and $V^*$: $Q_a^*=
\alpha P_a V^* +R_a$~\citep{bertsekas1996neuro}.

\subsection{LP formulation of dynamic programming with Q-function}

Motivated by this observation, we propose to consider the modified LP form
\begin{align}
&p^*_Q:= \min_{V\in {\mathbb R}^{|{\cal S}|},Q\in {\mathbb R}^{|{\cal S}||{\cal A}|}} \eta^T V\quad {\rm s.t.}\quad Q_a\le V,\quad \alpha P_a V + R_a = Q_a,\quad a
\in {\cal A},\label{eq:DP-LP-form2}
\end{align}
where $Q_a\in {\mathbb R}^{|{\cal S}|},a\in {\cal A}$. Compared
to~\eqref{eq:DP-LP-form}, the transition matrix $P_a$ and the inequality symbol are decoupled in~\eqref{eq:DP-LP-form2}.
To simplify the notation,
define the augmented vector
\begin{align*}
&Q:= \begin{bmatrix}
Q_1 \\
\vdots \\
Q_{|{\cal A}|} \\
\end{bmatrix} \in {\mathbb R}^{|{\cal S}||{\cal A}|}.
\end{align*}
Then, the LP form~\eqref{eq:DP-LP-form2} can be compactly
rewritten by
\begin{align}
&p^*_Q := \mathop {\rm min}_{V\in {\mathbb R}^{|{\cal S}|},Q\in {\mathbb R}^{|{\cal S}||{\cal A}|}}\eta^T
V\quad {\rm s.t.}\quad Q \le ({\bf 1}_{|{\cal A}|} \otimes I_{|{\cal
S}|} )V,\quad \alpha PV + R = Q.\label{eq:LP-form-Q-learning}
\end{align}

Since introducing the additional equality
constraints in~\eqref{eq:DP-LP-form2} does not affect the solution
$V^*$, we can easily prove that the optimal solution $V^*$
of~\eqref{eq:DP-LP-form2} is identical to that
of~\eqref{eq:DP-LP-form}.
\begin{lemma}\label{lemma:equivalence}
\begin{enumerate}
\item $(Q_a^*,V^*)_{a\in {\cal A}}$ is an optimal solution to the
LP~\eqref{eq:LP-form-Q-learning} if and only if $V^*$ is an
optimal solution to~\eqref{eq:DP-LP-form} and $Q_a^*=R+\alpha P_a
V^*,a \in {\cal A}$.

\item The optimal solution, $(Q_a^*,V^*)_{a\in {\cal A}}$,
to~\eqref{eq:LP-form-Q-learning} is unique.
\end{enumerate}
\end{lemma}
\begin{proof}
The first statement is trivial, and the second statement can be
directly proved using the first result.
\end{proof}

If $(V^*,Q^*)$ is an optimal solution to the LP, then $V^*$ is the
optimal value function, and $Q^*$ is the corresponding optimal
Q-factor. Once the optimal solution, $V^*,Q_a^*,a\in {\cal A}$, of~\eqref{eq:DP-LP-form2} is obtained, then $V^*$ is the optimal value function, and $Q_a^*,a\in {\cal A}$, is the optimal Q-function. Therefore, the optimal policy can be obtained using the primal solution, $Q_a^*,a\in {\cal A}$, as in the classical Q-learning.
Moreover, it can be recovered from the optimal dual solution as well. To study its dual problem, introduce the Lagrangian multipliers, $\lambda:=[\lambda_1^T,\cdots,\lambda_{|{\cal A}|}^T]^T$, for the inequality constraints, $\mu:=[\mu_1^T,\cdots,\mu_{|{\cal A}|}^T]^T$, for the equality constraints, and consider the Lagrangian function
\begin{align}
&L_I(Q,V,\lambda,\mu):=\eta^T V + \mu^T (\alpha PV + R - Q) +
\lambda^T (Q - ({\bf 1}_{|{\cal A}|} \otimes I_{|{\cal S}|}
)V)\label{eq:Lagrangian-LP-form-Q-learning}
\end{align}
for~\eqref{eq:LP-form-Q-learning}. Similarly to the original LP, the min-max problem satisfies
\begin{align}
&\min_{(V,\,Q)\in {\mathbb R}^{|{\cal S}|} \times {\mathbb R}^{|{\cal S}||{\cal A}|}} \max_{(\lambda,\,\mu) \in {\mathbb R}_+^{|{\cal S}||{\cal A}|} \times {\mathbb R}^{|{\cal S}||{\cal A}|}} L_I(Q,V,\lambda,\mu)\nonumber\\
&=\max_{(\lambda,\,\mu) \in {\mathbb R}_+^{|{\cal S}||{\cal A}|} \times {\mathbb R}^{|{\cal S}||{\cal A}|}} \min_{(V,\,Q)\in {\mathbb R}^{|{\cal S}|}\times {\mathbb R}^{|{\cal S}||{\cal A}|}} L_I
(Q,V,\lambda,\mu).\label{eq:saddle-point-1}
\end{align}
According to~\cref{proposition:saddle-point-relation}, there exists a saddle point $(V^*,Q^*,\lambda^*,\mu^*)$ satisfying $L(V^*,Q^*,\lambda,\mu)\le L(V,Q,\lambda,\mu)\le L(V,Q,\lambda^*,\mu^*),\forall (V,Q,\lambda,\mu) \in {\cal X} \times {\cal Y}$ with ${\cal X}= {\mathbb R}^{|{\cal S}|}\times {\mathbb R}^{|{\cal S}||{\cal A}|}$ and ${\cal Y} = {\mathbb R}_+^{|{\cal S}||{\cal A}|} \times {\mathbb R}^{|{\cal S}||{\cal A}|}$. In addition, its corresponding dual problem is given by
\begin{align}
&d_Q^*:= \max_{\mu \in {\mathbb R}^{|{\cal S}||{\cal A}|},\lambda\ge 0} \quad \mu^T R,\quad {\rm s.t.}\quad \eta+ \alpha P^T\mu -({\bf 1}_{|{\cal A}|} \otimes I_{|{\cal S}|})^T \lambda= 0,\quad \lambda=\mu.\label{eq:dual-LP-form-Q-learning}
\end{align}
We can prove that the dual optimal solution $(\tilde
\lambda^*,\tilde \mu^*)$ is $(\tilde\lambda^*,\tilde\mu^*)=(\lambda^*,\lambda^*)$, where $\lambda^*$ is the optimal dual solution in~\cref{lemma:uniqueness-of-dual-LP-DP}.
\begin{lemma}
The unique optimal solution $(\tilde \lambda^*,\tilde \mu^*)$ of
the dual problem~\eqref{eq:dual-LP-form-Q-learning} is $(\tilde
\lambda^*,\tilde\mu^*)=(\lambda^*,\lambda^*)$, where $\lambda^*$
is the optimal dual solution
in~\cref{lemma:uniqueness-of-dual-LP-DP}.
\end{lemma}
\begin{proof}
Since $(\tilde \lambda^*,\tilde \mu^*)$ is feasible, $\tilde \mu^*
= \tilde \lambda^*$ by the constraint $\mu=\lambda$
in~\eqref{eq:dual-LP-form-Q-learning}. Plugging $\mu=\lambda$ into
$\lambda$ in~\eqref{eq:dual-LP-form-Q-learning}, it is reduced
to~\eqref{eq:dual-LP}. Therefore, $\tilde \lambda^* = \lambda^*$,
where $\lambda^*$ is the solution of~\eqref{eq:dual-LP}. Since
$\lambda^*$ is unique by~\cref{lemma:uniqueness-of-dual-LP-DP}, so
is $(\tilde\lambda^*,\tilde\mu^*)=(\lambda^*,\lambda^*)$ as well. This
completes the proof.
\end{proof}

\subsection{Modified LP formulation of dynamic programming with Q-function}
In this subsection, in order to develop a model-free algorithm based on samples from arbitrary state-action distributions to solve the saddle point problem~\eqref{eq:saddle-point-1}, we
introduce another modified but equivalent LP form
\begin{align}
&p^*_Q := \min_{V,Q} \eta^T
V\quad {\rm s.t.}\quad Q \le ({\bf 1}_{|{\cal A}|} \otimes I_{|{\cal
S}|})V,\quad \alpha MPV + MR = MQ,\label{eq:LP-form-Q-learning2}
\end{align}
where
\begin{align*}
&M:= \begin{bmatrix}
   M_1 & & \\
    & \ddots & \\
    & & M_{|{\cal A}|}\\
\end{bmatrix}.
\end{align*}
$M_a,a\in {\cal A}$, is a diagonal matrix with strictly positive elements.
The diagonal elements of $M_a$ is the state distribution when the action $a\in {\cal A}$ is taken. Since $M$ is nonsingular, the above LP has the same solutions as those in~\eqref{eq:LP-form-Q-learning}.
\begin{proposition}\label{proposition:primal-solution-identical}
The optimal solution of~\eqref{eq:LP-form-Q-learning} is identical
to that of~\eqref{eq:LP-form-Q-learning2}.
\end{proposition}
\begin{proof}
Let $(Q_a^*,V^*)_{a\in {\cal A}}$ and $(\tilde Q_a^*,\tilde
V^*)_{a\in {\cal A}}$ be the optimal solution of the
LPs~\eqref{eq:LP-form-Q-learning}
and~\eqref{eq:LP-form-Q-learning2}, respectively. Multiplying
$\alpha PV^* + R\le Q^*$ from the left by $M$, we have $\alpha
MPV^* + MR \le MQ^*$. Therefore, $(Q_a^*,V^*)_{a\in {\cal A}}$ is
a feasible solution of~\eqref{eq:LP-form-Q-learning2}, and $\eta^T
V^*\le\eta^T \tilde V^*$. Similarly, one can prove the converse
$\eta^T V^*\ge\eta^T \tilde V^*$, and thus, $\eta^T
V^*=\eta^T\tilde V^*$. In addition, the feasible set
of~\eqref{eq:LP-form-Q-learning} is identical to the feasible set
of~\eqref{eq:LP-form-Q-learning2}. Having the same objectives and
feasible sets, LPs~\eqref{eq:LP-form-Q-learning}
and~\eqref{eq:LP-form-Q-learning2} have the identical solution
set. This completes the proof.
\end{proof}

To study its dual problem, introduce the Lagrangian multipliers, $\lambda:=[\lambda_1^T,\cdots,\lambda_{|{\cal A}|}^T]^T$, for the inequality constraints, $\mu:=[\mu_1^T,\cdots,\mu_{|{\cal A}|}^T]^T$, for the equality constraints, and consider the Lagrangian function
\begin{align*}
&L_M(Q,V,\lambda,\mu ) = \eta^T V + \mu^T M(\alpha PV + R - Q) +
\lambda^T (Q - ({\bf{1}}_{|{\cal A}|}  \otimes I_{|{\cal S}|} )V).
\end{align*}
Note that when setting $M = I_{|{\cal S}||{\cal A}|}$ in $L_M$, denoted by $L_I$, it reduces to the Lagrangian
function~\eqref{eq:Lagrangian-LP-form-Q-learning}
for~\eqref{eq:LP-form-Q-learning}, i.e., $L_I(Q,V,\lambda ,\mu ) =
\eta^T V + \mu^T (\alpha PV + R - Q) + \lambda^T (Q -
({\bf{1}}_{|{\cal A}|} \otimes I_{|{\cal S}|} )V)$. Then, the
optimal solution can be obtained by solving the saddle point
problem
\begin{align}
&\min_{(V,\,Q) \in {\mathbb R}^{|{\cal S}|}\times {\mathbb R}^{|{\cal S}||{\cal A}|}} \max_{(\lambda,\,\mu) \in {\mathbb R}_+^{|{\cal S}||{\cal A}|}\times {\mathbb R}^{|{\cal S}||{\cal A}|}} L_M(Q,V,\lambda,\mu)\nonumber\\
&=\max_{(\lambda,\,\mu)\in {\mathbb R}_+^{|{\cal S}||{\cal A}|}\times {\mathbb R}^{|{\cal S}||{\cal A}|}} \min_{(V,\,Q) \in {\mathbb R}^{|{\cal S}|}\times {\mathbb R}^{|{\cal S}||{\cal A}|}}
L_M(Q,V,\lambda,\mu).\label{eq:saddle-point-2}
\end{align}

According to~\citet[Prop.~2.6.1,pp.~132]{bertsekas2003convex}, there exists a saddle point $(V^*,Q^*,\lambda^*,\mu^*)$ satisfying
\begin{align*}
&L_M(V^*,Q^*,\lambda,\mu)\le L_M(V,Q,\lambda,\mu)\le L_M(V,Q,\lambda^*,\mu^*),\quad\forall (V,Q,\lambda,\mu) \in {\cal X} \times {\cal Y}
\end{align*}
with ${\cal X}= {\mathbb R}^{|{\cal S}|}\times {\mathbb R}^{|{\cal S}||{\cal A}|}$ and ${\cal Y} = {\mathbb R}_+^{|{\cal S}||{\cal A}|} \times {\mathbb R}^{|{\cal S}||{\cal A}|}$. In addition, the corresponding dual problem is
\begin{align}
&d_Q^*= \max_{(\lambda,\mu) \in {\mathbb R}_+^{|{\cal S}||{\cal A}|}\times {\mathbb R}^{|{\cal S}||{\cal A}|}}\quad \mu^T MR\quad {\rm s.t.}\quad\eta- ({\bf 1}_{|{\cal A}|} \otimes I_{|{\cal S}|})^T\lambda
+\alpha P^T M\mu= 0,\quad M\mu= \lambda.\label{eq:dual-LP-form-Q-learning2}
\end{align}

\cref{proposition:primal-solution-identical} suggests that the
primal optimal solutions of~\eqref{eq:LP-form-Q-learning}
and~\eqref{eq:LP-form-Q-learning2} are identical. However, it may not be the case for the
dual optimal solutions. In the next
proposition, we study expressions of the dual solution.
\begin{proposition}\label{proposition:dual-solution}
Let $(\mu^*,\lambda^*)$ and $(\tilde\mu^*,\tilde\lambda^*)$ be the
optimal solutions of~\eqref{eq:dual-LP-form-Q-learning} and~\eqref{eq:dual-LP-form-Q-learning2}, respectively. Then, $\tilde \mu^*
=M^{-1}\lambda^*$ and $\tilde\lambda^*=\lambda^*$.
\end{proposition}
\begin{proof}
Let $(V^*,Q^*)$ and $(\mu^*,\lambda^*)$ be the optimal solutions of
the primal problem~\eqref{eq:LP-form-Q-learning} and dual
problem~\eqref{eq:dual-LP-form-Q-learning}, respectively. Then,
they are the solution of the saddle point
problem~\eqref{eq:saddle-point-1}. Similarly, if $(\tilde
V^*,\tilde Q^*)$ and $(\tilde\mu^*,\tilde\lambda^*)$ are the optimal
solutions of the primal problem~\eqref{eq:LP-form-Q-learning2} and
dual problem~\eqref{eq:dual-LP-form-Q-learning2}, respectively,
then they are the solution of the saddle point
problem~\eqref{eq:saddle-point-2}. We will prove that
$(Q^*,V^*,\lambda^*,M^{-1} \mu^*)$ is a saddle point of
$L_M(\cdot,\cdot,\cdot,\cdot)$. Since $(Q^*,V^*,\lambda^*,\mu^*)$
is a saddle point of $L_I(\cdot,\cdot,\cdot,\cdot)$, we have
\begin{align*}
&L_I(Q^*,V^*,\lambda^*,\mu^*)\le L_I(Q^*,V^*,\lambda,\mu),\quad
\forall \lambda \in {\mathbb R}_+^{|{\cal S}||{\cal A}|},\mu\in {\mathbb R}^{|{\cal S}||{\cal A}|},
\end{align*}
and equivalently, $L_M (Q^*,V^*,\lambda^*,M^{-1}\mu^*)\le
L_M(Q^*,V^*,\lambda ,M^{-1}\mu)$ for all $\lambda$ and $\mu$.
Since $M$ is nonsingular, this is equivalent to $L_M
(Q^*,V^*,\lambda^*,M^{-1}\mu^*)\le L_M
(Q^*,V^*,\lambda,\mu),\forall\lambda,\mu$, and by the definition
of the saddle point, one concludes that $(Q^*,V^*,\lambda^*,M^{-1}
\mu^*)$ is a saddle point
of~\eqref{eq:saddle-point-2}. Therefore, $(\lambda^*,M^{-1}\mu^*)$ is the dual optimal solution of~\eqref{eq:dual-LP-form-Q-learning2}, i.e., $\tilde \mu^*
=M^{-1}\lambda^*$ and $\tilde\lambda^*=\lambda^*$. This completes the proof.
\end{proof}

\cref{proposition:dual-solution} suggests that the optimal dual solution for $\lambda$ is identical to that of the original LP~\eqref{eq:DP-LP-form}. Therefore, the optimal policy can be recovered from the dual solution as well as the primal solution. Based on this observation, we can also establish bounds for the solutions to the modified saddle point problem. Note that such bounds allow us to restrict the saddle point problem to compact domains, which is essential for analyzing the convergence of primal-dual type algorithms. For this aim, we find bounds on the solutions in the next lemma.
\begin{lemma}\label{lemma:bound-of-solutions}
Let $(Q^*,V^*)$ and $(\lambda^*,\mu^*)$ be the optimal primal and
dual solutions solving~\eqref{eq:LP-form-Q-learning2}
and~\eqref{eq:dual-LP-form-Q-learning2}, respectively. In addition, let
$\zeta>0$ be a real number less than or equal to any diagonal
element of $M$. Then, we have
\begin{enumerate}
\item $\| Q_a^* \|_\infty\le \|V^*\|_\infty \le
\frac{\sigma}{1-\alpha},\quad \forall a \in {\cal A}$

\item $\|Q_a^*\|_2\le \| V^* \|_2 \le \frac{\sqrt{|{\cal S}|} \sigma
}{1-\alpha},\quad \forall a \in {\cal A}$

\item $\| \lambda^* \|_\infty\le \|\lambda^*\|_2 \le
\|\lambda^*\|_1 \le \frac{\| \eta \|_1}{1-\alpha}$.

\item $\|\mu^*\|_\infty\le \|\mu^*\|_2 \le \|\mu^*\|_1 \le
\frac{\| \eta \|_1}{\zeta(1-\alpha)}$.
\end{enumerate}
\end{lemma}
\begin{proof}
The constraints in~\eqref{eq:LP-form-Q-learning2} imply $0\le Q_a^*
\le V^* $ for all $a \in {\cal A}$. In combination with this
result, the first and second statements follow
by~\cref{lemma:bound-lemma}. Since $\lambda^*$ is identical to the
optimal dual variable of the original dual
problem~\eqref{eq:dual-LP-form-Q-learning}
by~\cref{proposition:dual-solution}, the third result follows
from~\cref{lemma:bound-lemma}. Let $\tilde\lambda^*$ be the
optimal solution of the dual
problem~\eqref{eq:dual-LP-form-Q-learning}.
By~\cref{proposition:dual-solution}, we prove $\|\mu^*\|_2\le
\|\mu^*\|_1 = \| M^{-1}\tilde \lambda^*\|_1\le \|M^{-1}
\|_{1,1} \| \tilde\lambda^* \|_1$, where $\|\cdot
\|_{1,1}$ is the induced matrix norm associated with the vector 1-norm. Using~\cref{lemma:bound-lemma}, one
obtains
\begin{align*}
&\|\mu^*\|_2\le \|\mu^*\|_\infty \le \|\mu^*\|_1 = \|M^{-1}
\tilde\lambda^*\|_1 \le \|M^{-1}\|_{1,1} \|\tilde\lambda^*\|_1
\le \|M^{-1}\|_{1,1}\frac{\|\eta\|_1}{1-\alpha} \le \frac{\|\eta
\|_1}{\zeta (1-\alpha)},
\end{align*}
where the inequalities $\|\cdot \|_\infty\le \|\cdot\|_2
\le \|\cdot \|_1$ are used. This completes
the proof.
\end{proof}

\subsection{SPD Q-learning algorithm}

In this subsection, we develop a stochastic primal-dual Q-learning algorithm that solves the saddle point problem~\eqref{eq:saddle-point-2} in the previous subsection. Based on~\cref{lemma:bound-of-solutions}, define the compact convex sets
\begin{align*}
&{\cal V}:= \left\{ v \in {\mathbb R}^{|{\cal S}|}:v \ge
0,\|v\|_\infty\le \frac{\sigma}{1-\alpha} \right\},\quad {\cal L}:
= \left\{ \lambda\in {\mathbb R}^{|{\cal S}||{\cal A}|}:\lambda \ge
0,\| \lambda \|_\infty \le \frac{\| \eta \|_1}{1-\alpha}
\right\},\\
&{\cal M}:= \left\{ \mu\in {\mathbb R}^{|{\cal S}||{\cal A}|} :\mu
\ge 0,\|\mu\|_\infty \le \frac{\| \eta \|_1}{\zeta (1-\alpha)}
\right\},\quad \Xi:=\left\{ \lambda \in {\mathbb R}^{|{\cal
S}||{\cal A}|}:\sum_{a\in {\cal A}}{\lambda_a}\ge \eta \right\},
\end{align*}
which satisfy $V^*\in {\cal V},Q_a^*\in {\cal V}$ for all $a\in
{\cal A}$, $\lambda^*\in {\cal L} $, and $\mu^*\in {\cal M}$, where $\zeta>0$ is a real number less than or equal to any diagonal element of $M$. The construction of the compact sets is similar to~\citet*{chen2016stochastic}, so we omit the details here for brevity. Interested readers are referred to~\citet*{chen2016stochastic} for details. Then, the domain of each variable of the saddle point problem in~\eqref{eq:saddle-point-2} can be confined into a smaller compact set as follows:
\begin{align}
&\min_{(V,\,Q)\in {\cal V}\times {\cal V}^{|{\cal A}|}}\max_{(\lambda,\,\mu)\in ({\cal L} \cap \Xi)\times {\cal M}} L_M(Q,V,\lambda,\mu)=
\max_{(\lambda,\,\mu)\in ({\cal L} \cap \Xi)\times {\cal M}}\min_{(V,\,Q)\in {\cal V}\times {\cal V}^{|{\cal A}|}} L_M(Q,V,\lambda,\mu).\label{eq:saddle-point-3}
\end{align}
Note that solutions of~\eqref{eq:saddle-point-3}
and~\eqref{eq:saddle-point-2} are identical. The Markov decision problem now reduces to solving~\eqref{eq:saddle-point-3}. If the discounted MDP model is known, then it can be solved by using the (deterministic) primal-dual algorithm~\citep{nedic2009subgradient}
\begin{align*}
&Q_{k+1}=\Pi_{{\cal V}^{|{\cal A}|}} [Q_k-\gamma_k \nabla_{Q} L_{M}(Q_k,V_k,\lambda_k,\mu_k)],\\
&V_{k+1}=\Pi_{\cal V} [V_k-\gamma_k\nabla_V L_{M}(Q_k,V_k,\lambda_k,\mu_k)],\\
&\lambda_{k+1}=\Pi_{{\cal L} \cap \Xi} [\lambda_k+\gamma_k\nabla_\lambda L_{M}(Q_k,V_k,\lambda_k,\mu_k)],\\
&\mu_{k+1}=\Pi_{\cal M} [\mu_k+\gamma_k\nabla_\mu
L_{M}(Q_k,V_k,\lambda_k,\mu_k)],
\end{align*}
where the gradients of $L_{M}$ with respect to the primal
variables, $Q,V$, and the dual variables, $\lambda,\mu$, are
\begin{align*}
&\nabla_Q L_M (Q,V,\lambda,\mu)=\lambda-M\mu= N|{\cal S}||{\cal A}|\lambda  - M\mu,\\
&\nabla_V L_M (Q,V,\lambda,\mu) = \eta  - ({\bf 1}_{|{\cal A}|}
\otimes I_{|{\cal S}|})^T \lambda+ \alpha P^T M\mu\nonumber\\
&\quad\quad\quad\quad\quad\quad\quad\quad=H|{\cal S}|\eta - ({\bf 1}_{|{\cal A}|}\otimes I_{|{\cal S}|})^T
N|{\cal S}||{\cal A}|\lambda  + \alpha P^T M\mu,\\
&\nabla_\lambda L_M(Q,V,\lambda,\mu)=Q - ({\bf 1}_{|{\cal A}|}
\otimes I_{|{\cal S}|})V = N|{\cal S}||{\cal A}|Q-
N|{\cal S}||{\cal A}|({\bf 1}_{|{\cal A}|} \otimes I_{|{\cal S}|})V\\
&\nabla_\mu L_M(Q,V,\lambda,\mu) =\alpha MPV + MR - MQ,
\end{align*}
where
\begin{align*}
&N:= \begin{bmatrix}
   \frac{1}{|{\cal A}|}H & & \\
    &  \ddots  &  \\
    &  & \frac{1}{|{\cal A}|}H  \\
\end{bmatrix},\quad H:= \begin{bmatrix}
   \frac{1}{|{\cal S}|} & & \\
    &  \ddots  &  \\
    & & \frac{1}{|{\cal S}|} \\
\end{bmatrix}.
\end{align*}
We introduce the matrices, $N,H$, to radomize the gradients, $\nabla_Q L_M$, $\nabla_\mu L_M$, and reduce the computational complexity per each iteration. Assume that the discounted MDP model is unknown, but the trajectory, $(s_k,a_k)_{k=0}^\infty$, can be observed in real-time. Then, noisy estimates of partial derivatives of the Lagrangian function can be obtained from samples of state-action transitions. We can therefore apply the stochastic primal-dual algorithm to solve~\eqref{eq:saddle-point-3}.

Although it is common in RL literature to assume a stationary distribution, i.e., $M$ is constant, the algorithm can also handle the case that $M$ is time-varying. In particular, let $\tau_{a,k}(s),s\in {\cal S},a \in {\cal A}$, be the probability
that the current state and action are $(s,a)$ at time $k$,
respectively, and define the corresponding vectors and matrices
\begin{align}
&\tau_{a,k}:=\begin{bmatrix}
\tau_{a,k}(1)\\
\vdots\\
\tau_{a,k}(|{\cal S}|)\\
\end{bmatrix}\in {\mathbb R}^{|{\cal S}|},\quad M_{a,k}:= \begin{bmatrix}
   \tau_{a,k}(1) & & \\
    & \ddots  & \\
    & & \tau_{a,k}(|{\cal S}|)\\
\end{bmatrix} \in {\mathbb R}^{|{\cal S}|\times |{\cal S}|},\nonumber\\
& M_k:= \begin{bmatrix}
   M_{1,k} & & \\
    & \ddots & \\
    & & M_{|{\cal A}|,k}\\
\end{bmatrix} \in {\mathbb R}^{|{\cal S}||{\cal A}| \times |{\cal S}||{\cal
A}|}.\label{eq:M-matrix-def}
\end{align}
The diagonal elements of $M_k$ represent the probability measure on
the state-action space $(s,a)\in {\cal S} \times {\cal A}$ at time
$k$. One can think of $(M_k)_{k=0}^{\infty}$ as a deterministic
infinite sequence of matrices given a priori. To proceed further, we adopt the
following assumptions.
\begin{assumption}\label{assump:time-varying-distribution}
Throughout the paper, we assume that there exists a real number
$\zeta>0$ such that $\tau_{a,k}(s)\ge\zeta$ for all $k \ge 0$ and
$(s,a) \in {\cal S} \times {\cal A}$. Moreover, there exists a
positive diagonal matrix $M_\infty$ such that $\lim_{k\to \infty}
M_k=M_\infty$.
\end{assumption}
\begin{example}\label{ex:ex1}
\cref{assump:time-varying-distribution} includes the case that the
behavior policy $\theta$ is fixed, but the state distribution at time
$k$ did not reach a stationary distribution. Another case is the
behavior policy itself is time-varying. In particular, consider
any stochastic policy $\theta_s \in \Delta_{|{\cal A}|}$ such that
$\theta_s(a)>0,\forall s \in {\cal S},a \in {\cal A}$ and the
corresponding state-to-state transition matrix $P_\beta$. Assume that the initial state
distribution $v_0\in \Delta_{|{\cal S}|}$ at time $k=0$ is
given. Then, the state distribution at time $k$ is
$v_k=(P_\beta^T)^k v_0$. In addition, assume that the MDP is ergodic
under $\beta$, i.e., there exists a stationary distribution
$v_\infty \in \Delta_{|{\cal S}|}$ such that $\lim_{k\to\infty} v_k=
\lim_{k\to\infty} (P_\beta^T)^k v_0= v_\infty$ and each element of
the stationary distribution vector $v_\infty$ is positive (e.g.,
the Markov chain is irreducible and aperiodic~\citep[Theorem
2.13.2.]{resnick2013adventures}). Then, the state-action
distribution at time $k$ is $\tau_{a,k}(s)=v_k(s)\theta_a(s),s \in
{\cal S},a\in {\cal A}$ and its stationary distribution is
$\tau_{a,\infty}(s) = v_\infty(s)\theta_a(s),s \in {\cal S},a \in
{\cal A}$. This distribution and the corresponding matrix $M_k$
satisfy~\cref{assump:time-varying-distribution}.
\end{example}

Replacing $M$ with $M_k$ in $L_M$, the corresponding Lagrangian
function $L_{M_k}$ is given by
\begin{align*}
&L_{M_k}(Q,V,\lambda,\mu)=\eta^T V + \mu^T M_k(\alpha PV + R - Q)
+ \lambda^T (Q - ({\bf 1}_{|{\cal A}|} \otimes I_{|{\cal S}|})V),
\end{align*}
where $M_k$ is changing for all $k \geq 0$. The corresponding primal-dual algorithm can be modified as follows:
\begin{align}
&Q_{k+1}=\Pi_{{\cal V}^{|{\cal A}|}}[Q_k-\gamma_k \nabla_{Q}
L_{M_k}(Q_k,V_k,\lambda_k,\mu_k)],\nonumber\\
&V_{k+1}=\Pi_{\cal V} [V_k-\gamma_k \nabla_V L_{M_k} (Q_k,V_k,\lambda_k,\mu_k)],\nonumber\\
&\lambda_{k+1}=\Pi_{{\cal L}\cap\Xi} [\lambda_k+\gamma_k\nabla_\lambda L_{M_k}(Q_k,V_k,\lambda_k,\mu_k)],\nonumber\\
&\mu_{k+1}=\Pi_{\cal M} [\mu_k +\gamma_k \nabla_\mu
L_{M_k}(Q_k,V_k,\lambda_k,\mu_k)].\label{eq:primal-dual-iteration}
\end{align}
Using the distributions in $M_k$, the gradients
in~\eqref{eq:primal-dual-iteration} can be replaced with their
stochastic estimations. The corresponding algorithm is
summarized in~\cref{algo:primal-dual-method}. Since $M_k$ is time-varying, it is not clear whether or not the iterates converge, and if yes, how fast the convergence speed is and to
which solution they converge. In the next section, we provide
answers to these questions. Finally, we close this section by formally introducing two ways to obtain a suboptimal policy from~\cref{algo:primal-dual-method}.
\begin{definition}[Primal and dual policies]\label{def:primal-dual-policy}
The {\em primal policy} associated with~\cref{algo:primal-dual-method} is defined as the deterministic policy
\begin{align}
&\hat\pi_T^p(s):={\rm argmax}_{a\in {\cal A}} \hat Q_{a,T}(s)\label{eq:primal-policy}
\end{align}
recovered from the primal variable $\hat Q_{a,T}$. The {\em dual policy} associated with~\cref{algo:primal-dual-method} is defined as the stochastic policy
\begin{align}
&\hat\pi_T^d(s):= \begin{bmatrix}
\frac{\hat\lambda_{1,T}^T e_s}{\sum_{a'\in {\cal A}}{\hat\lambda_{a',T}^T e_s}} & \cdots & \frac{\hat\lambda_{|{\cal A}|,T}^T e_s}{\sum_{a'\in {\cal A}}{\hat\lambda_{a',T}^Te_s}}\\
\end{bmatrix}^T\in \Delta_{|{\cal A}|}\label{eq:dual-policy}
\end{align}
recovered from the dual variable $\hat\lambda_T$.
\end{definition}

\begin{algorithm}[h!]
\caption{SPD Q-learning algorithm}
\begin{algorithmic}[1]
\State Initialize $V^{(0)}:{\cal S}\to
\left[0,\frac{\sigma}{1-\alpha}\right],Q^{(0)}:{\cal S} \times
{\cal A} \to \left[0,\frac{\sigma}{1-\alpha} \right],\lambda^{(0)}
:{\cal S}\times {\cal A}\to \left[0,\frac{\| \eta \|_1
}{(1-\alpha)} \right]$, $\mu^{(0)}:{\cal S}\times {\cal A} \to
\left[ 0,\frac{\| \eta \|_1}{\zeta (1-\alpha)} \right]$.

\For{$k =0,1,\ldots,T-1$}

\State Observe $(s_k,a_k,s_{k+1},\hat r_{s_k s_{k+1} a_k})$ from
the environment, where $(s_k,a_k) \sim \tau_k$.

\State Uniformly sample $\hat s_k \sim U(S),\hat a_k\sim U(A)$.

\State Update the primal iterates by
\begin{align*}
Q_{a_k}^{(k+1/4)}=& Q_{a_k}^{(k)}-\gamma_k [-e_{s_k} e_{s_k}^T \mu_{a_k}^{(k)}],\\
Q_{\hat a_k}^{(k+1/2)}=& Q_{\hat a_k}^{(k+1/4)}-\gamma_k [|{\cal
S}||{\cal A}|e_{\hat s_k} e_{\hat s_k}^T \lambda_{\hat a_k}^{(k)}],\\
Q_a^{(k+1/2)}=& Q_a^{(k)},\quad a \in {\cal A}\backslash \{a_k,\hat a_k\},\\
V^{(k+1/2)}=&V^{(k)}-\gamma_k [e_{\hat s_k} e_{\hat s_k}^T |{\cal
S}|\eta - e_{\hat s_k} e_{\hat s_k}^T |{\cal S}||{\cal
A}|\lambda_{\hat a_k}^{(k)}+ \alpha e_{s_{k+1}} e_{s_k}^T \mu_{a_k}^{(k)}].
\end{align*}

\State Update the dual iterates by
\begin{align*}
\lambda_{\hat a_k}^{(k+1/2)}=& \lambda_{\hat a_k}^{(k)} + \gamma_k
[|{\cal S}||{\cal A}| e_{\hat s_k} e_{\hat s_k}^T Q_{\hat a_k}^{(k)}- |{\cal S}||{\cal A}|e_{\hat s_k} e_{\hat s_k }^T V^{(k)}],\\
\lambda_a^{(k+1/2)}=& \lambda_a^{(k)},\quad a \in {\cal
A}\backslash \{\hat a_k \},\\
\mu_{a_k}^{(k+1/2)}=& \mu_{a_k}^{(k)}+\gamma_k [\alpha e_{s_k}
e_{s_{k+1}}^T V^{(k)} + e_{s_k}\hat r_{s_k s_{k+1} a_k}- e_{s_k} e_{s_k}^T Q_{a_k}^{(k)}],\\
\mu_a^{(k+1/2)}=&\mu_a^{(k)},\quad a \in {\cal A}\backslash \{a_k\}.
\end{align*}

\State Project the iterates onto the convex sets
\begin{align*}
&V^{(k+1)}= \Pi_{\cal V} (V^{(k+1/2)}),\quad
Q_{a_k}^{(k+1)}=\Pi_{\cal V} (Q_{a_k}^{(k+1/2)}),\\
&Q_{\hat a_k}^{(k+1)}=\Pi_{\cal V} (Q_{\hat a_k}^{(k+1/2)}),\quad \lambda^{(k+1)}=\Pi_{{\cal L}\cap \Xi} (\lambda^{(k+1/2)}),\\
&\mu^{(k+1)}=\Pi_{\cal M}(\mu^{(k+1/2)}).
\end{align*}

\EndFor

\State {\bf Output}: Averaged iterates $\hat Q_T=
\frac{1}{T}\sum_{k=0}^{T-1}{Q^{(k)}}$, $\hat V_T=
\frac{1}{T}\sum_{k=0}^{T-1}{V^{(k)}}$, and $\hat\lambda_T =
\frac{1}{T}\sum_{k=0}^{T-1}{\lambda^{(k)}}$.
\end{algorithmic}
\label{algo:primal-dual-method}
\end{algorithm}

\section{Main Result}\label{section:main-result}
In this section, we summarize main results of this paper, including the convergence
of~\cref{algo:primal-dual-method}. To achieve this goal, basic assumptions are summarized below.
\begin{assumption}
The step-size sequence $(\gamma_k)_{k=0}^\infty$ is
non-increasing, and $\lim_{k\to\infty}\gamma_k=0$.
\end{assumption}
\begin{assumption}\label{assump:distribution-convergence}
There exists a non-increasing sequence $(\beta_k)_{k=0}^\infty$
such that
\begin{align*}
&\| M_k^{-1}-M_{k+1}^{-1}\|_2 \le \beta_k,\quad\forall k\ge 0
\end{align*}
and $\lim_{k\to\infty}\beta_k=0$.
\end{assumption}

We first introduce an example to prove the validity of~\cref{assump:distribution-convergence}.
\begin{example}\label{ex:ex2}
We prove that~\cref{ex:ex1}
satisfies~\cref{assump:distribution-convergence}. According to the definition in~\eqref{eq:M-matrix-def}, the
corresponding matrix, $M_k$, has diagonal entries
$\tau_{a,\,k}(s)=v_k(s)\theta_s(a),s \in {\cal S},a \in {\cal
A}$. Then, the diagonal entries of $M_k^{-1}-M_{k+1}^{-1}$ are
$\frac{1}{v_k(s)\theta_s(a)} - \frac{1}{v_{k+1}(s)\theta_s(a)},a \in
{\cal A},s \in {\cal S}$, and
\begin{align*}
&\|M_k^{-1} - M_{k+1}^{-1}\|_2
=\sqrt{\lambda_{\max}((M_k^{-1}-M_{k+1}^{-1})^T(M_k^{-1}-M_{k+1}^{-1}))}\\
&= \sqrt{\max_{a\in {\cal A},s\in {\cal
S}}\left(\frac{1}{v_k(s)\theta_s(a)} -\frac{1}{v_{k+1}(s)\theta_s(a)}\right)^2} = \max_{a\in {\cal A},s \in {\cal S}}\left( \frac{v_{k+1}(s)-v_k(s)}{v_k(s)v_{k+1}(s)\theta_a(s)}\right)\\
&\le\frac{\max_{s\in {\cal S}}(v_{k+1}(s)-v_k(s))} {\min_{a\in
{\cal A},s\in {\cal S}}(v_k(s)v_{k+1}(s)\theta_s(a))}\le
\frac{\max_{s\in {\cal S}}(v_{k+1}(s)-v_k(s))}{\min_{a\in {\cal
A},s \in {\cal S}}(v_k(s)v_{k+1}(s)\theta_s(a)\theta_s(a))}\\
&\le\zeta^{-2}\max_{s\in {\cal S}}(v_{k+1}(s)-v_k(s)),
\end{align*}
which converges to zero as $k\to\infty$. Therefore, one can set
$\beta_k=\zeta^{-2}\max_{s\in {\cal S}}(v_{k+1}(s)-v_k(s))$.
\end{example}

For simplicity, define the vectors and sets
\begin{align*}
&x:=\begin{bmatrix}
   Q\\
   V\\
\end{bmatrix},\quad y: = \begin{bmatrix}
   \lambda \\
   \mu \\
\end{bmatrix},\quad {\cal X}:= {\cal V}^{|{\cal A}|}\times{\cal V},\quad{\cal Y}:=({\cal L}\cap\Xi)\times{\cal
M},
\end{align*}
where $x\in {\cal X}$ and $y \in {\cal Y}$ collect the primal and dual variables,
respectively. Then, the Lagrangian function can be written
compactly by
\begin{align*}
L_{M_k}(Q,V,\lambda,\mu) &=\eta^T V+\mu^T M_k(\alpha PV + R -
Q)+\lambda^T(Q-({\bf 1}_{|{\cal A}|}\otimes I_{|{\cal S}|})V)\\
&= f(x)+x^T A_k y+b_k^T y\\
&=:L_{M_k}(x,y),
\end{align*}
where
\begin{align*}
&f(x):=\eta^T V,\quad A_k=\begin{bmatrix}
   I & -({\bf 1}_{|{\cal A}|}  \otimes I_{|{\cal S}|})  \\
   -M_k & \alpha M_k P\\
\end{bmatrix}^T,\quad b_k= \begin{bmatrix}
 0\\
 M_k R\\
\end{bmatrix}.
\end{align*}
The primal-dual updates in~\eqref{eq:primal-dual-iteration} are
written as
\begin{align}
&x_{k+1}=\Pi_{\cal X}(x_k-\gamma_k(\nabla_x
L_{M_k}(x_k,y_k)+\varepsilon_k)),\label{eq:promai-dual-subgrad1}\\
&y_{k+1}=\Pi_{\cal Y}(y_k+\gamma_k(\nabla_y
L_{M_k}(x_k,y_k)+\xi_k)),\label{eq:promai-dual-subgrad2}
\end{align}
where $\nabla_x L_{M_k}(x,y)$ and $\nabla_y L_{M_k}(x,y)$ are
gradients of $L_{M_k}(x,y)$ with respect to $x$ and $y$, respectively,
and $(\varepsilon_k,\xi_k)$ are (possibly dependent) random
variables with zero mean. In particular, by taking the expectation in the iterates of~\cref{algo:primal-dual-method}, we can prove that they can be expressed as~\eqref{eq:promai-dual-subgrad1} and~\eqref{eq:promai-dual-subgrad2} with the unbiased stochastic gradients $\nabla_x L_{M_k}(x_k,y_k)+\varepsilon_k$ and $\nabla_y L_{M_k}(x_k,y_k)+\xi_k$, respectively. Based on the above definitions, the notion of the duality gap of the constrained saddle point problem $\min_{x\in {\cal X}} \max_{y\in {\cal Y}} L_I(x,y)=\max_{y\in {\cal Y}} \min_{x\in {\cal X}}L_I(x,y)$ is defined as follows.
\begin{definition}[Pseudo duality gap]
The pseudo duality gap of the constrained saddle point problem $\min_{x\in {\cal X}} \max_{y\in {\cal Y}} L_I(x,y)=\max_{y\in {\cal Y}} \min_{x\in {\cal X}}L_I(x,y)$ at any point $(x,y)\in {\cal X} \times {\cal Y}$ is defined as
\begin{align*}
&D(x,y):=L_I(x,y^*)-L_I(x^*,y),
\end{align*}
where $(x^*,y^*)\in {\cal X} \times {\cal Y}$ is the solution of the saddle point problem.
\end{definition}

Our first main result establishes the convergence rate of the (pseudo) duality gap. To this end, one needs to assume that the sequence $(\beta_k)_{k=0}^\infty$ satisfies a certain condition. In particular, we assume that there
exists a real number $\beta_0>0$ such that $\beta_k=\beta_0/(k+1),k\ge 0$. In the following, we provide an example which satisfies the assumption.
\begin{example}\label{ex:ex3}
In this example, we consider~\cref{ex:ex2} again and prove that $\beta_k$ in~\cref{ex:ex2} is upper bounded by $\beta_0/(k+1)$ with some real number $\beta_0>0$. For simplicity, assume that $P_\theta$ is
diagonalizable so that it has $|{\cal S}|$ independent left
eigenvectors $u_1,u_2,\ldots,u_{|{\cal S}|}$ and the corresponding
left eigenvalues $\lambda_1,\lambda_2,\ldots,\lambda_{|{\cal
S}|}$, respectively. If we define
\begin{align*}
&U:= \begin{bmatrix}
   u_1 & \cdots & u_{|{\cal S}|}\\
\end{bmatrix},\quad \Sigma:= \begin{bmatrix}
   \lambda_1 & & \\
   & \ddots & \\
   & & \lambda_{|{\cal S}|}\\
\end{bmatrix},
\end{align*}
then $P_\theta= U\Sigma U^{-1}$ by the similarly
transformation~\citep[pp.~114]{gentle2007matrix}. Since $P_\theta$ is
a row stochastic matrix, we can enumerate the eigenvalues as $1 =
|\lambda_1|> |\lambda_2|\ge |\lambda_3|\ge |\lambda_4|\ge\cdots\ge
|\lambda_{|{\cal S}|}|$ and $\lambda_1=1$
(see~\citet[pp.~306]{gentle2007matrix}). Then, one can prove the following results.
\begin{proposition}\label{proposition:ex3}
(a) There exists some real number $c >0$ such that $\beta_k \leq c |\lambda_2|^k$; (b) There exists some real number $d>0$ such that $|\lambda_2|^k\le d /(k+1)$ for all $k\geq 0$.
\end{proposition}

The proof of~\cref{proposition:ex3} is given in {\bf Appendix~E}. Combining this with~\eqref{eq:eq14} yields $\beta_k \le \frac{c \cdot d}{k+1},\forall k \ge 0$. Therefore, we can set $\beta_k=\beta_0/(k+1)$ with $\beta_0=c\cdot d$.
\end{example}
Throughout the paper, $\hat x_T$ and $\hat y_T$ are defined as
\begin{align}
&\hat x_T = \frac{1}{T}\sum_{k=0}^{T-1}{x_k},\quad \hat
y_T=\frac{1}{T}\sum_{k=0}^{T-1}{\bar M_k y_k},\label{eq:def:hat-x-y}\\
&x_k:=\begin{bmatrix}
   Q_k\\
   V_k\\
\end{bmatrix},\quad y_k: = \begin{bmatrix}
   \lambda_k \\
   \mu_k \\
\end{bmatrix},\nonumber
\end{align}
where $\bar M_k := \begin{bmatrix} I_{|{\cal S}||{\cal A}|} & 0\\ 0 & M_k\\ \end{bmatrix}$. The first main result establishes the convergence of the duality gap $D(\hat x_T,\hat y_T)$ at point $(\hat x_T,\hat y_T)\in {\cal X}\times {\cal Y}$. Note that according to the definitions $\hat Q_T=
\frac{1}{T}\sum_{k=0}^{T-1}{Q^{(k)}}$, $\hat V_T=
\frac{1}{T}\sum_{k=0}^{T-1}{V^{(k)}}$, and $\hat\lambda_T =
\frac{1}{T}\sum_{k=0}^{T-1}{\lambda^{(k)}}$, if we define $\hat\mu_T =
\frac{1}{T}\sum_{k=0}^{T-1}{M_k\mu^{(k)}}$, then $\hat x_T = \begin{bmatrix} \hat Q_T\\\hat V_T\\\end{bmatrix}$ and $\hat y_T = \begin{bmatrix} \hat \lambda_T\\\hat \mu_T\\\end{bmatrix}$.
\begin{theorem}\label{proposition:duality-gap-convergence}
Assume that $\gamma_k=\gamma_0/\sqrt{k+1},k \ge 0$ with some
$\gamma_0>0$, $\beta_k=\beta_0/(k+1),k\ge 0$ for any real
number $\beta_0>0$, and let $\eta = \frac{\sigma}{|{\cal S}|} \bf{1}_{|{\cal S}|}$. For any $\varepsilon \in (0,1)$ and $\delta\in (0,1/e)$, where $e$ is the Euler's number, if
\begin{align*}
&T\ge\kappa\frac{\sigma^2|{\cal S}|^4|{\cal A}|^4}{\zeta^4(1-\alpha)^4}\frac{1}{\varepsilon^2}\ln \left(\frac{1}{\delta}\right)
\end{align*}
then with probability at least $1-\delta$, we have $D(\hat x_T,\hat y_T)\le\varepsilon$, where $\kappa:=\max\{\kappa_1,\kappa_2\}$,
\begin{align*}
&\kappa_1:=\left(\frac{12+4\beta_0}{\zeta^2|{\cal S}|^2|{\cal A}|^2\gamma_0}+26\gamma_0\right)^2,\\
&\kappa_2:=(2184+416\sqrt{26})\gamma_0^2 + (1066+416\sqrt{26})\gamma_0+832+16\sqrt{26}.
\end{align*}
\end{theorem}

The proof of~\cref{proposition:duality-gap-convergence} will be presented in the next section. If we run~\cref{algo:primal-dual-method}, then with the number of iterations, $T$, at least $\kappa \frac{\sigma^4 |{\cal S}|^4 |{\cal A}|^4}{\zeta^4(1-\alpha)^4}\frac{1}{\varepsilon^2}\ln \left(\frac{1}{\delta}\right)$, the duality gap, $D(\hat x_T,\hat y_T)$, is less than or equal to $\varepsilon$ with probability $1-\delta$, meaning that $D(\hat x_T,\hat y_T)\to 0$ almost surely as $T\to\infty$ and that the primal and dual solutions almost surely converge to the true ones solving~\eqref{eq:saddle-point-3} as $T\to \infty$. In what follows, we establish the fact that an $\varepsilon$-suboptimal policy can be recovered within a finite number of iterations, where the meaning of $\varepsilon$-suboptimal policy will be defined soon.
\begin{theorem}\label{proposition:dual-policy-convergence}
Assume that $\gamma_k=\gamma_0/\sqrt{k+1},k \ge 0$ with some
$\gamma_0>0$, $\beta_k=\beta_0/(k+1),k\ge 0$ for any real
number $\beta_0>0$, and let $\eta = \frac{\sigma}{|{\cal S}|} \bf{1}_{|{\cal S}|}$. For any $\varepsilon \in (0,1)$ and $\delta\in (0,1/e)$, where $e$ is the Euler's number, if
\begin{align}
&T\ge\kappa \frac{\sigma^2|{\cal S}|^6|{\cal A}|^4}{\zeta^4(1-\alpha)^6}\frac{1}{\varepsilon^2}\ln\left( \frac{1}{\delta}\right),\label{eq:complexity-bound}
\end{align}
then with probability at least $1-\delta$, we have $\|V^*-V^{\hat
\pi_T^d}\|_\infty\le\varepsilon$, where $\hat \pi_T^d$ is the dual policy defined in~\cref{def:primal-dual-policy}.
\end{theorem}

If we run~\cref{algo:primal-dual-method}, then after the number of iterations, $T$, at least $T\ge\kappa \frac{\sigma^2|{\cal S}|^6|{\cal A}|^4}{\zeta^4(1-\alpha)^6}\frac{1}{\varepsilon^2}\ln\left( \frac{1}{\delta}\right)$, the policy, $\hat \pi_T^d$, obtained from the dual iterates is $\varepsilon$-suboptimal with probability at least $1-\delta$ in the sense that the distance, $\|V^*-V^{\hat \pi_T^d} \|_\infty$, between the optimal value function and the value function corresponding to $\hat \pi_T^d$ is less than or equal to $\varepsilon$. The complexity bound in~\eqref{proposition:duality-gap-convergence} is not better than those of existing methods, for instance, ${\cal O}\left(\frac{|{\cal S}|^4|{\cal A}|^2\sigma^2}{(1-\alpha)^6\varepsilon^2}\ln\left(\frac{1}{\delta}\right)\right)$ of the SPD-RL algorithm~\citep[Theorem~4]{chen2016stochastic} and ${\cal O}\left(\frac{|{\cal S}||{\cal A}|}{(1-\alpha)^8\varepsilon^4}\ln \left(\frac{1}{\delta}\right)\right)$ of the delayed Q-learning~\citep{strehl2009reinforcement}. The increased complexity can be viewed as a cost to pay for its off-policy and online learning ability. Lastly, \cref{proposition:duality-gap-convergence} suggests that the SPD Q-learning guarantees the convergence with a certain convergence rate even when the state-action distribution under a certain behavior policy is time-varying but sub-linearly converges to a stationary distribution as stated in~\cref{assump:distribution-convergence}. To the author's best knowledge, this seems to be the first convergence analysis in the context of  time-varying state-action distributions for off-policy RL. Details of the convergence proofs are given in the next section.

\section{Convergence Analysis}\label{section:convergence}

The main goal of this section is to provide proofs of the convergence
results of~\cref{algo:primal-dual-method}. Define the $\sigma$-field
\begin{align*}
{\cal F}_k:=\sigma(&\varepsilon_0,\ldots,\varepsilon_{k-1},\xi_0,\ldots,\xi_{k-1},x_0,\ldots,x_k,y_0,\ldots,y_k)
\end{align*}
related to all random variables of the algorithm until time $k$.
The following lemma introduces basic iterate
relations~\citep{nedic2009subgradient}.
\begin{lemma}[Basic iterate
relations~\citep{nedic2009subgradient}]\label{lemma:basic-iterate-relations1}
Let the sequences $(x_k,y_k)_{k=0}^\infty$ be generated by the SPD
algorithm in~\eqref{eq:promai-dual-subgrad1}
and~\eqref{eq:promai-dual-subgrad2}. Then, we have:
\begin{enumerate}
\item For any $x \in {\mathbb R}^{|{\cal S}|} \times {\mathbb R}^{|{\cal S}||{\cal A}|}$ and for all $k \geq 0$,
\begin{align*}
{\mathbb E}[\|x_{k+1}-x\|_2^2 |{\cal F}_k] &\le \| x_k-x
\|_2^2 + \gamma_k^2 {\mathbb E}[\| \nabla _x L_{M_k}
(x_k,y_k)+\varepsilon_k \|_2^2 |{\cal F}_k]\\
&-2\gamma_k(L_{M_k}(x_k,y_k)-L_{M_k}(x,y_k)).
\end{align*}

\item For any $y\in {\mathbb R}_+^{|{\cal S}||{\cal A}|} \times {\mathbb R}^{|{\cal S}||{\cal A}|}$ and for all $k \geq 0$,
\begin{align*}
{\mathbb E}[\| y_{k+1}-y \|_2^2 |{\cal F}_k ] &\le \| y_k-y \|_2^2
+ \gamma_k^2 {\mathbb E}[\| \nabla _y L_{M_k} (x_k ,y_k ) + \xi_k
\|_2^2 |{\cal F}_k ]\\
&+2\gamma_k(L_{M_k}(x_k,y_k)-L_{M_k}(x_k,y)).
\end{align*}
\end{enumerate}
\end{lemma}
\begin{proof}
The result can be obtained by the iterate relations
in~\citet[Lemma~3.1]{nedic2009subgradient} and taking the
expectations.
\end{proof}
To prove the convergence, it is essential to establish the
boundedness of the stochastic gradients in~\eqref{eq:promai-dual-subgrad1}
and~\eqref{eq:promai-dual-subgrad2}. Particular bounds are given in the next result.
\begin{lemma}\label{lemma:K1-K2}
We have
\begin{align*}
&\|\nabla_x L_{M_k}(x_k,y_k)+\varepsilon_k\|_2 \le
\frac{\sqrt{13}|{\cal S}||{\cal A}| \|\eta\|_1
}{\zeta(1-\alpha)} =:K_1,\\
&\|\nabla_y L_{M_k}(x_k,y_k)+\xi_k\|_2 \le \frac{\sqrt{13} |{\cal
S}||{\cal A}|\sigma}{1-\alpha}=:K_2.
\end{align*}
\end{lemma}
\begin{proof}
See~{\bf Appendix~A}.
\end{proof}

For any $x \in {\cal X}$ and $y \in {\cal Y}$, define
\begin{align*}
&{\cal E}_k^{(1)}(x):=\|x_k - x\|_2^2, \quad {\cal E}_k^{(2)}(y):= \|y_k -
y\|_2^2.
\end{align*}
In the next proposition, we derive a bound on the duality gap.
\begin{proposition}[Duality gap bound]\label{lemma:basic-iterate-relations2}
If we define
\begin{align*}
&H_k(x):=\frac{1}{2\gamma_k}({\cal E}_k^{(1)}(x) - {\mathbb
E}[{\cal E}_{k+1}^{(1)}(x)|{\cal F}_k]),\quad R_k(y):= \frac{1}{2\gamma_k}({\cal E}_k^{(2)}(y)-{\mathbb
E}[{\cal E}_{k+1}^{(2)}(y)|{\cal F}_k]),
\end{align*}
then, we have
\begin{align}
&D({\hat x}_T,{\hat y}_T)\le
\frac{1}{T}\sum_{k=0}^{T-1}{R_k(\bar M_k^{-1} y^*)}+
\frac{1}{T}\sum_{k=0}^{T-1} {H_k(x^*)} +
\frac{1}{T}\sum_{k=0}^{T-1}
{\frac{\gamma_k}{2}(K_1^2+K_2^2)}\label{eq:duality-gap-bound}
\end{align}
with probability one, where $(x^*,y^*) \in {\cal X}\times {\cal Y}$ is the primal-dual solution of the saddle point problem $\min_{x\in {\cal X}} \max_{y\in {\cal Y}} L_I(x,y)=\max_{y\in {\cal Y}} \min_{x\in {\cal X}}L_I(x,y)$.
\end{proposition}
\begin{proof}
We use ${\mathbb E}[\| \nabla_x L_{M_k}(x_k,y_k)+\varepsilon_k
\|_2^2 |{\cal F}_k] \le K_1^2$ and rearrange terms
in~\cref{lemma:basic-iterate-relations1} to have
\begin{align}
&L_{M_k}(x_k,y_k)-L_{M_k}(x,y_k)\le \underbrace
{\frac{1}{2\gamma_k}({\cal E}_k^{(1)}(x) - {\mathbb E}[{\cal
E}_{k+1}^{(1)}(x)|{\cal F}_k])}_{=:H_k(x)} + \frac{\gamma_k}{2}
K_1^2,\quad \forall x \in {\mathbb R}^{|{\cal S}||{\cal A}|} \times {\mathbb R}^{|{\cal S}|},\label{eq:basic1}\\
&-\underbrace {\frac{1}{2\gamma_k}({\cal E}_k^{(2)}(y)-{\mathbb
E}[{\cal E}_{k+1}^{(2)}(y)|{\cal F}_k ])}_{=:R_k(y)} -
\frac{\gamma_k}{2} K_2^2 \le L_{M_k}(x_k,y_k)-L_{M_k}(x_k,y),\quad
\forall y \in {\mathbb R}_+^{|{\cal S}||{\cal A}|} \times {\mathbb R}^{|{\cal S}||{\cal A}|}.\label{eq:basic2}
\end{align}
By plugging $\bar M_k^{-1}y \in {\mathbb R}_+^{|{\cal S}||{\cal A}|} \times {\mathbb R}^{|{\cal S}||{\cal A}|}$ with $y \in {\cal Y}$ into $y$ in~\eqref{eq:basic2}, adding
these relations over $k=0,\ldots,T-1$, dividing by $T$, and
rearranging terms, we have
\begin{align}
&-\frac{1}{T}\sum_{k=0}^{T-1}{R_k(\bar M_k^{-1}y)}-\frac{1}{T}\sum_{k=0}^{T-1}
{\frac{\gamma_k}{2}K_2^2}\le
\frac{1}{T}\sum_{k=0}^{T-1}{(L_{M_k}(x_k,y_k)-L_I(x_k,y))},\quad
\forall y \in {\cal Y}.\label{eq:eq1}
\end{align}
Similarly, we have from~\eqref{eq:basic1}
\begin{align}
&\frac{1}{T}\sum_{k=0}^{T-1}{(L_{M_k}(x_k,y_k)-L_{M_k}(x,y_k))}\le
\frac{1}{T}\sum_{k=0}^{T-1}{H_k(x)}+\frac{1}{T}\sum_{k=0}^{T-1}
{\frac{\gamma_k}{2} K_1^2} ,\quad \forall x \in {\cal X}.\label{eq:eq2}
\end{align}
Using the convexity of $L_{M_k}$ with respect to the first
argument, it follows from~\eqref{eq:eq1} that
\begin{align}
&- \frac{1}{T}\sum_{k=0}^{T-1}{R_k (\bar M_k^{-1}y)}
-\frac{1}{T}\sum_{k=0}^{T-1}{\frac{\gamma_k}{2}K_2^2}\le
\frac{1}{T}\sum_{k=0}^{T - 1} {L_{M_k}(x_k,y_k)}-L_I(\hat
x_T,y),\quad \forall y \in {\cal Y}.\label{eq:eq3}
\end{align}
Similarly, using the concavity of $L_I$ with respect to the second
argument, it follows from~\eqref{eq:eq2} that
\begin{align}
&\frac{1}{T}\sum_{k=0}^{T-1} {L_{M_k}(x_k,y_k)}-L_I(x,\hat y_T)
\le \frac{1}{T}\sum_{k=0}^{T-1}{H_k(x)}+
\frac{1}{T}\sum_{k=0}^{T-1}{\frac{\gamma_k}{2}K_1^2},\quad \forall
x \in {\cal X},\label{eq:eq4}
\end{align}
where we use the definition of $\hat y_T$ in~\eqref{eq:def:hat-x-y} to change $L_{M_k}$ to $L_I$. Multiplying both sides of~\eqref{eq:eq4} by $-1$ and adding it
with~\eqref{eq:eq3} yields
\begin{align*}
&L_I(\hat x_T,y)-L_I(x,\hat y_T)\le \frac{1}{T}\sum_{k=0}^{T-1}
{R_k (\bar M_k^{-1} y)}+\frac{1}{T}\sum_{k=0}^{T-1}
{H_k(x)}+\frac{1}{T}\sum_{k=0}^{T-1}{\frac{\gamma_k}{2}(K_1^2+K_2^2)},\\
&\forall x \in {\cal X},y\in {\cal Y}.
\end{align*}
Letting $x= x^* \in {\cal X}$ and $y= y^* \in {\cal Y}$, we obtain
\begin{align*}
&0 \le D(\hat x_T,\hat y_T)\le
\frac{1}{T}\sum_{k=0}^{T-1}{R_k(\bar M_k^{-1} y^*)}+
\frac{1}{T}\sum_{k=0}^{T-1}{H_k(x^*)} +\frac{1}{T}\sum_{k=0}^{T-1}
{\frac{\gamma_k}{2}(K_1^2+K_2^2)},
\end{align*}
and this completes the proof.
\end{proof}

To proceed, we rearrange terms in~\eqref{eq:duality-gap-bound} to
have
\begin{align}
&D(\hat x_T,\hat y_T)\le \frac{1}{T}\Phi_1(x^*)
+\frac{1}{T}\Phi_2(y^*)+\frac{1}{T}{\cal
G}_T+\frac{K_1^2+K_2^2}{2}\frac{1}{T}\sum_{k=0}^{T-1}{\gamma_k},\label{eq:duality-gap-bound2}
\end{align}
where
\begin{align*}
&\Phi_1(x):=\sum_{k=0}^{T-1}{\frac{1}{2\gamma_k}({\cal
E}_k^{(1)}(x)-{\cal E}_{k+1}^{(1)}(x))},\quad
\Phi_2(y):=\sum_{k=0}^{T-1}{\frac{1}{2\gamma_k}({\cal
E}_k^{(2)}(\bar M_k^{-1}y)-{\cal E}_{k+1}^{(2)}(M_k^{-1}y))},\\
&{\cal G}_T:=\sum_{k=0}^{T-1}{\frac{1}{2\gamma_k}({\cal
E}_{k+1}^{(1)}(x^*) + {\cal E}_{k+1}^{(2)}(\bar M_k^{-1} y^*)- {\mathbb
E}[{\cal E}_{k+1}^{(1)}(x^*)|{\cal F}_k]- {\mathbb E}[{\cal
E}_{k+1}^{(2)}(\bar M_k^{-1} y^*)|{\cal F}_k])},
\end{align*}
and ${\cal E}_k^{(1)}(x):=\|x_k - x\|_2^2$, ${\cal E}_k^{(2)}(y):= \|y_k -
y\|_2^2$. In the next result, we derive bounds on the terms $\Phi_1(x^*)$ and $\Phi_2(y^*)$.
\begin{lemma}\label{lemma:Phi-bounds}
We have
\begin{align*}
&\Phi_1(x^*)\le\frac{2\sigma^2 |{\cal S}|}{(1-\alpha)^2}\frac{1}{\gamma_{T-1}},\\
&\Phi_2(y^*)\le \frac{1}{\gamma_{T-1}}\frac{4\| \eta \|_1^2}{\zeta^4
(1-\alpha)^2}+\frac{2 \|\eta
\|_1^2}{\zeta^3(1-\alpha)^2}\sum_{k=1}^{T-1}{\frac{\beta_{k-1}}{\gamma_{k-1}}}.
\end{align*}
\end{lemma}
\begin{proof}
See {\bf Appendix~B}.
\end{proof}

Combining~\eqref{eq:duality-gap-bound2}
with~\cref{lemma:Phi-bounds} and using the definitions of $K_1$ and
$K_2$ in~\cref{lemma:K1-K2}, one gets
\begin{align}
D(\hat x_T,\hat y_T)&\le\frac{2\sigma^2\zeta^4 |{\cal S}|+4\|\eta\|_1^2}{\zeta^4
(1-\alpha)^2}\frac{1}{T\gamma_{T-1}} + \frac{2
\|\eta\|_1^2}{\zeta^3(1-\alpha)^2}\frac{1}{T}\sum_{k=1}^{T-1}
{\frac{\beta_{k-1}}{\gamma_{k-1}}}\nonumber\\
&+\frac{13}{2}\frac{|{\cal S}|^2 |{\cal
A}|^2(\|\eta\|_1^2+\sigma^2)}{\zeta^2(1-\alpha)^2}\,\frac{1}{T}\sum_{k=0}^{T-1}{\gamma_k}
+\frac{1}{T}{\cal G}_T.\label{eq:duality-gap-bound3}
\end{align}
From~\eqref{eq:duality-gap-bound3}, one observes that under
certain conditions, the right-hand side converges to zero except
for the last term $\frac{1}{T}{\cal G}_T$. For instance, if
$\lim_{T\to\infty} \frac{1}{T\gamma_{T-1}}=0$, $\lim_{T\to\infty}
\frac{1}{T}\sum_{k=1}^{T-1}{\frac{\beta_{k-1}}{\gamma_{k-1}}}=0$,
and $\lim_{T\to\infty} \frac{1}{T}\sum_{k=0}^{T-1}{\gamma_k}=0$,
then
\begin{align}
\limsup_{T\to\infty}D(\hat x_T,\hat y_T)
\le\limsup_{T\to\infty}\frac{1}{T}{\cal G}_T\label{eq:explain-8}
\end{align}
In particular, if we set $(\gamma_k)_{k=0}^\infty$ to be $\gamma_k
=\gamma_0/\sqrt{k+1},k\ge 0$, for some real number $\gamma_0>0$ and
$\beta_k=\beta_0/(k+1),k\ge 0$, with a real number $\beta_0>0$,
then~\eqref{eq:explain-8} holds true. In the next proposition, we simplify the right-hand side of~\eqref{eq:duality-gap-bound3}.
\begin{proposition}\label{proposition:duality-bound}
If $\gamma_k=\gamma_0/\sqrt{k+1},k\ge 0$ and
$\beta_k=\beta_0/(k+1),k \ge 0$ for a real number
$\gamma_0>0,\beta_0>0$, then
\begin{align*}
&D(\hat x_T,\hat y_T)\le \frac{C_0}{\sqrt
T}+\frac{1}{T}{\cal G}_T,
\end{align*}
where
\begin{align*}
&C_0:=\frac{2\sigma^2\zeta^4 |{\cal S}|+4\|\eta\|_1^2+2\zeta
\|\eta\|_1^2 \beta_0}{\zeta^4(1-\alpha)^2\gamma_0}+\gamma_0
\frac{13}{2}\frac{|{\cal S}|^2 |{\cal
A}|^2(\|\eta\|_1^2+\sigma^2)}{\zeta^2(1-\alpha)^2}.
\end{align*}
\end{proposition}
\begin{proof}
The result can be proved by estimating convergence rates of the
upper bounds on the three terms,
$1/(T\gamma_{T-1}),\sum_{k=0}^{T-1}{\gamma_k/T}$, and
$\sum_{k=1}^{T-1}{\beta_{k-1}/(\gamma_{k-1} T)}$,
in~\eqref{eq:duality-gap-bound3}. Plugging
$\gamma_k=\gamma_0/\sqrt{k+1}$ into the first two terms, we have
$1/(T\gamma_{T-1})=1/(\sqrt T \gamma_0)$ and
\begin{align*}
&\frac{1}{T}\sum_{k=0}^{T-1}{\gamma_k}=\frac{\gamma_0}{T}\sum_{k=1}^T
{\frac{1}{\sqrt k}} \le \frac{\gamma_0}{T}\int_0^T {\frac{1}{\sqrt
t}dt}= \frac{\gamma_0 \sqrt T}{T} = \frac{\gamma_0}{\sqrt T}.
\end{align*}
Similarly, using the definitions
$\gamma_k=\gamma_0/\sqrt{k+1},\beta_k=\gamma_0/\sqrt{k+1}$ in the
last term leads to
\begin{align*}
&\frac{1}{T}\sum_{k=1}^{T-1}{\frac{\beta_{k-1}}{\gamma_{k-1}}}=\frac{\beta_0}{T\gamma_0}\sum_{k=0}^{T-2}
{\frac{\sqrt{k+1}}{k+1}}=\frac{\beta_0}{T\gamma_0}\sum_{k=1}^{T-1}
{\frac{1}{\sqrt k}}  \le \frac{\beta_0}{T\gamma_0}\int_0^{T-1}
{\frac{1}{\sqrt t}dt} = \frac{\beta_0\sqrt{T-1}}{T\gamma_0} \le
\frac{\beta_0}{\gamma_0 \sqrt T}.
\end{align*}
Combining these results with~\eqref{eq:duality-gap-bound3}, we
have the desired result.
\end{proof}

From~\cref{proposition:duality-bound}, we have $\limsup_{T\to\infty}D(\hat x_T,\hat y_T)
\le\limsup_{T\to\infty}\frac{1}{T}{\cal G}_T$. Now, we focus on the last term, ${\cal G}_T/T$,
in~\eqref{eq:duality-gap-bound3}. Compared to the other terms
in~\eqref{eq:duality-gap-bound3}, proving the
boundedness of ${\cal G}_T/T$ is not straightforward. Therefore, we will use the
properties of Martingale sequence and the concentration
inequalities~\citep{bercu2015concentration} to prove the convergence as
in~\citet{chen2016stochastic}. To do so,
first define ${\cal E}_k :={\cal E}_k^{(1)}(x^*)+{\cal
E}_k^{(2)}(\bar M_k^{-1} y^*)$, where ${\cal E}_k^{(1)}(x):=\|x_k - x\|_2^2$, and ${\cal E}_k^{(2)}(y):= \|y_k-y\|_2^2$. Then, ${\cal G}_T$ is written by
\begin{align*}
&{\cal G}_T:=\sum_{k=0}^{T-1}{\frac{1}{2\gamma_k}({\cal
E}_{k+1}-{\mathbb E}[{\cal E}_{k+1}|{\cal F}_k])},
\end{align*}
where ${\cal F}_k:=\sigma(\varepsilon_0,\ldots,\varepsilon_{k-1},\xi_0,\ldots,\xi_{k-1},x_0,\ldots,x_k,y_0,\ldots,y_k)$. By the construction of ${\cal G}_T$, one easily proves that $({\cal
G}_T)_{T=0}^\infty$ with ${\cal G}_0=0$ is a Martingale, i.e.,
${\mathbb E}[{\cal G}_{t+1}|{\cal F}_t]={\cal G}_t$ holds as
\begin{align*}
{\mathbb E}[{\cal G}_{T+1}|{\cal F}_T]&={\mathbb E}\left[ \left.
\sum_{k=0}^T {\frac{1}{2\gamma_k}({\cal E}_{k+1}- {\mathbb
E}[{\cal E}_{k+1}|{\cal F}_k])} \right|{\cal F}_T
\right]\\
&={\mathbb E}\left[ \left. \frac{1}{2\gamma_k}({\cal
E}_{T+1}-{\mathbb E}[{\cal E}_{T+1}|{\cal F}_T])\right|{\cal F}_T\right] +{\mathbb E}\left[ \left. \sum_{k=0}^{T-1} {\frac{1}{2\gamma_k}({\cal E}_{k+1}- {\mathbb E}[{\cal E}_{k+1}|{\cal F}_k])} \right|{\cal F}_T\right]\\
&= {\mathbb E} \left[ \left. \sum_{k=0}^{T-1}
{\frac{1}{2\gamma_k}({\cal E}_{k+1}- {\mathbb E}[{\cal E}_{k+1}
|{\cal F}_k])}\right|{\cal F}_T\right]={\cal G}_T.
\end{align*}

The next step is to use the Berstein inequality for
Martingales~\citep{freedman1975tail,fan2012hoeffding,bercu2015concentration}
to prove the convergence of ${\cal G}_T/T$,
in~\eqref{eq:duality-gap-bound3}. For completeness of the presentation, the Berstein inequality is formally
stated in the following lemma.
\begin{lemma}[Berstein inequality for
Martingales~\citep{bercu2015concentration}]\label{lemma:Bernstein-inequality}
Let $({\cal G}_T)_{T=0}^\infty$ be a square integrable martingale
such that ${\cal G}_0=0$. Assume that $\Delta {\cal G}_T\le
b,\forall T \ge 1$ with probability one, where $b>0$ is a real
number and $\Delta{\cal G}_T$ is the Martingale difference defined
as $\Delta{\cal G}_T={\cal G}_T-{\cal G}_{T-1},T \ge 1$. Then, for
any $\varepsilon\in [0,b]$ and $a>0$,
\begin{align}
&{\mathbb P} \left[ \frac{1}{T}{\cal G}_T\ge\varepsilon
,\frac{1}{T} \langle {\cal G}\rangle_T  \le a\right]
\le\exp\left(-\frac{T\varepsilon^2}{2(a+b\varepsilon/3)}\right),\label{eq:Berstein}
\end{align}
where
\begin{align*}
&\langle {\cal G} \rangle_T:=\sum_{k=0}^{T-1}{{\mathbb E}[({\cal
G}_{k+1}-{\cal G}_k)^2|{\cal F}_k]}=\sum_{k=0}^{T-1}{{\mathbb
E}[\Delta {\cal G}_{k+1}^2|{\cal F}_k]}.
\end{align*}
\end{lemma}
To apply~\cref{lemma:Bernstein-inequality}, we first prove that
the martingale difference for $T\ge 1$
\begin{align*}
&\Delta{\cal G}_T :={\cal G}_T-{\cal G}_{T-1}=
\frac{1}{2\gamma_{T-1}}({\cal E}_T-{\mathbb E}[{\cal E}_T|{\cal
F}_{T-1}])
\end{align*}
is bounded by a real number $b>0$.
\begin{lemma}\label{lemma:bound1}
We have $\Delta {\cal G}_{T+1}={\cal G}_{T+1}-{\cal
G}_T=\frac{1}{2\gamma_T}({\cal E}_{T+1}-{\mathbb E}[{\cal
E}_{T+1}|{\cal F}_T])\le b$ with probability one, where
\begin{align*}
&b=\frac{13\gamma_0 \|\eta\|_1+
13\gamma_0\zeta^2\sigma^2+16\sqrt{26}\sigma\zeta
\|\eta\|_1}{2}\frac{|{\cal S}|^2 |{\cal
A}|^2}{\zeta^2(1-\alpha)^2}.
\end{align*}
\end{lemma}
\begin{proof}
See~{\bf Appendix~C}.
\end{proof}

Similarly, we can prove that there exists a real number $a >0$ such that $\frac{1}{T} \langle {\cal G} \rangle_T\le a $ holds with probability one so as to remove the condition $\frac{1}{T} \langle {\cal G} \rangle_T\le a $ in~\eqref{eq:Berstein}.
\begin{lemma}\label{lemma:bound2}
$\frac{1}{T} \langle {\cal G} \rangle_T\le a $ holds with
probability one, where
\begin{align*}
&a=\frac{1}{4}\frac{(\gamma_0(13 \|\eta\|_1 +
4\sqrt{26}\sigma\zeta)\|\eta\|_1+13\gamma_0\zeta^2\sigma^2+4\sqrt{26}\sigma
\|\eta\|_1)^2 |{\cal S}|^4 |{\cal A}|^4}{\zeta^4(1-\alpha)^4}.
\end{align*}
\end{lemma}
\begin{proof}
See~{\bf Appendix~D}.
\end{proof}

From the series of results derived so far, we collected all useful ingredients to prove the convergence of~\cref{algo:primal-dual-method}. Now, details of the proof of~\cref{proposition:duality-gap-convergence} are given in the next subsection.

\subsection{Proof of~\cref{proposition:duality-gap-convergence}}
We apply the Bernstein inequality
in~\cref{lemma:Bernstein-inequality} with $a$ in~\cref{lemma:bound2}
and $b$ in~\cref{lemma:bound1} to prove
\begin{align*}
&{\mathbb P}\left[ \frac{1}{T}{\cal
G}_T\ge\beta\varepsilon,\frac{1}{T}\langle {\cal G}\rangle_T\le a
\right]= {\mathbb P}\left[ \frac{1}{T}{\cal G}_T \ge \beta
\varepsilon \right] \le \exp\left(
-\frac{T\beta^2\varepsilon^2}{2(a+b\beta\varepsilon/3)}\right)
\end{align*}
with any $\beta\in (0,1)$ and $\varepsilon>0$.
By~\cref{lemma:Bernstein-inequality}, for any $\delta\in (0,1)$,
$\exp\left(-\frac{T\beta^2\varepsilon^2}{2(a+b\beta\varepsilon/3)}\right)
\le \delta$ holds if and only if $T \ge
\frac{2(a+b\beta\varepsilon/3)}{\beta^2\varepsilon^2}\ln
(\delta^{-1})$. Therefore, if $\exp\left(-\frac{T\beta^2
\varepsilon^2}{2(a+b\beta\varepsilon/3)}\right) \le\delta$, then
with probability at least $1-\delta$, we have ${\cal
G}_T/T\le\beta\varepsilon$, which in combination
with~\cref{proposition:duality-bound} implies $D(\hat
x_T,\hat y_T)\le C_0/\sqrt T+\beta\varepsilon$.
With algebraic manipulations, one proves
$\Psi_1=\frac{2(a+b\beta\varepsilon/3)}{\beta^2\varepsilon^2}\ln(\delta^{-1})$.
Similarly, $C_0\le\varepsilon(1-\beta)$ holds if and only if $T\ge
\frac{C_0^2}{\varepsilon^2(1-\beta)^2}$. If $\delta\in (0,1/e)$,
then $\ln(1/\delta)\ge 1$, and the last inequality holds if $T\ge
\frac{C_0^2}{\varepsilon^2(1-\beta)^2}\ln (\delta^{-1})$. Plugging
$C_0$ in~\cref{proposition:duality-bound} into
$\frac{C_0^2}{\varepsilon^2(1-\beta)^2}\ln (\delta^{-1})$ and
after algebraic simplifications, one gets
$\Psi_2=\frac{C_0^2}{\varepsilon^2(1-\beta)^2}\ln(\delta^{-1})$. Therefore, if $T\ge\max\{\Psi_1,\Psi_2\}$, then with probability at least $1-\delta$, $D(\hat x_T,\hat y_T)\le\varepsilon$ holds. The desired conclusion is obtained by setting $\beta = 1/2$, $\eta = \frac{\sigma}{|{\cal S}|} \bf{1}_{|{\cal S}|}$, and after algebraic simplifications.

\subsection{Proof of~\cref{proposition:dual-policy-convergence}}

Lastly, the proof of~\cref{proposition:dual-policy-convergence} is given in this subsection. Before beginning the proof, we first derive an equivalent formulation of the duality gap $D$ in~\cref{proposition:duality-gap-convergence}.
\begin{lemma}\label{lemma:duality-gap-representation}
We have
\begin{align*}
&D(\hat x_T,\hat y_T)= L_I(\hat x_T,y^*)-L_I(x^*,\hat y_T)=\sum_{a'\in {\cal A}} {\hat \lambda_{a',T}^T}(I-\alpha P^{\hat\pi_T^d})(V^*-V^{\hat \pi_T^d}),
\end{align*}
where $V^{\hat\pi_T^d}$ is the value function corresponding to the dual policy~\eqref{eq:dual-policy}.
\end{lemma}
\begin{proof}
The proof is completed by the equalities
\begin{align}
&L_I(\hat x_T,y^*)-L_I(x^*,\hat y_T)\nonumber\\
&= L_I(\hat Q_T,\hat V_T,\lambda^*,\mu^*)-L_I(Q^*,V^*,\hat\lambda_T,\hat\mu_T)\nonumber\\
&=\eta^T\hat V_T+(\mu^*)^T(\alpha P\hat V_T+R-\hat
Q_T)+(\lambda^*)^T (\hat Q_T - ({\bf 1}_{|{\cal A}|}
\otimes I_{|{\cal S}|})\hat V_T)\nonumber\\
&-\eta^T V^*-\hat\mu_T^T(\alpha
PV^*+R-Q^*)-\hat\lambda_T^T(Q^*-({\bf 1}_{|{\cal A}|}\otimes
I_{|{\cal S}|})V^*)\nonumber\\
&=\eta^T(\hat V_T- V^*)+(\lambda^*)^T(\alpha P\hat V_T + R - ({\bf
1}_{|{\cal A}|} \otimes I_{|{\cal S}|})\hat
V_T)-\hat\lambda_T^T(Q^*-({\bf 1}_{|{\cal A}|}\otimes
I_{|{\cal S}|})V^*)\label{eq:explain-14}\\
&=\eta^T(\hat V_T - V^*)+\sum_{a \in {\cal A}}{(\lambda_a^*)^T}
(\alpha P_a \hat V_T + R_a -\hat V_T)-\sum_{a\in {\cal A}}{\hat\lambda_{a,T}^T}(Q_a^* - V^*)\nonumber\\
&=\eta^T(\hat V_T-V^*)+ \sum_{s\in {\cal S}}{\left(\sum_{a'\in
{\cal A}}{(\lambda_{a'}^*)^T e_s}\right)\sum_{a\in {\cal A}}
{\frac{(\lambda_a^*)^T e_s}{\sum_{a'\in {\cal
A}}{(\lambda_{a'}^*)^T e_s}}}(\alpha
e_s^T P_a \hat V_T + e_s^T R_a - e_s^T \hat V_T)}\nonumber\\
&-\sum_{s\in {\cal S}}{\left(\sum_{a'\in {\cal
A}}{\hat\lambda_{a',T}^T e_s} \right)\sum_{a\in {\cal A}}
{\frac{\hat\lambda_{a,T}^T e_s}{\sum_{a'\in {\cal A}}{\hat
\lambda_{a',T}^T e_s}}(e_s^T Q_a^* - e_s^T V^*)}}\nonumber\\
&=\eta^T (\hat V_T -V^*)+\sum_{a'\in {\cal A}}{(\lambda_{a'}^*)^T}
(\alpha P_{\pi^*}\hat V_T+R_{\pi^*}-\hat V_T)-\sum_{a'\in {\cal
A}}{\hat\lambda_{a',T}^T}(Q_{\hat\pi_T^d}^*  -V^*)\label{eq:explain-15}\\
&= \eta ^T (\hat V_T -V^*)+\sum_{a'\in {\cal
A}}{(\lambda_{a'}^*)^T}(\alpha P_{\pi^*} \hat V_T +V^* -\alpha
P_{\pi^*} V^* - \hat V_T)-\sum_{a'\in {\cal
A}}{\hat\lambda_{a',T}^T}(\alpha P_{\hat \pi_T^d} V^*+R_{\hat\pi_T^d}-V^*)\label{eq:explain-16}\\
&=\eta^T (\hat V_T -V^*)+\sum_{a'\in {\cal A}}{(\lambda_{a'}^*)^T}
(\alpha P_{\pi^*}(\hat V_T-V^*)-(\hat V_T-V^*))\nonumber\\
&-\sum_{a'\in {\cal A}}{\hat\lambda_{a',T}^T}(\alpha P_{\hat\pi_T^d}
V^*+V_{\hat\pi_T^d}-\alpha P_{\hat\pi_T^d} V_{\hat\pi_T^d}-V^*)\label{eq:explain-17}\\
&=\eta^T(\hat V_T-V^*)+\left(\sum_{a'\in {\cal
A}}{(\lambda_{a'}^*)^T}\right)(\alpha P_{\pi^*}-I)(\hat V_T - V^*)
+\left(\sum_{a'\in {\cal A}}{\hat\lambda_{a',T}^T}
\right)(I-\alpha P_{\hat\pi_T^d})(V^*-V_{\hat\pi_T^d})\nonumber\\
&=\eta^T(\hat V_T - V^*)-\eta^T(\hat V_T - V^*)+\left(\sum_{a'\in
{\cal A}}{\hat\lambda_{a',T}^T}\right)(I-\alpha P_{\hat\pi_T^d})(V^*
- V_{\hat\pi_T^d}),\label{eq:explain-18}
\end{align}
where $e_s\in {\mathbb R}^{|{\cal S}|},s\in {\cal S}$ is the
$s$-th basis vector (all components are $0$ except for the $s$-th
component which is $1$),~\eqref{eq:explain-14} is due to $\alpha
PV^*+R-Q^*$ and $\mu^* = \lambda^*$,~\eqref{eq:explain-15} is due to the relations
\begin{align*}
&P_{\pi^*}=\sum_{a\in {\cal A}}{\pi_a^*(s)e_s^T e_s P_a},\quad
R_{\pi^*}= \sum_{a\in {\cal A}}{\pi_a^*(s)e_s^T e_s R_a},\quad
Q_{\hat\pi_T^d}^*=\sum_{a \in {\cal A}} {\hat\pi_{a,T}^d(s)e_s^T e_s
Q_a^*}
\end{align*}
by the definitions~\eqref{eq:def:P-mu}
and~\eqref{eq:def:R-mu},~\eqref{eq:explain-16} is due to
$R_{\pi^*}=V^* -\alpha P_{\pi^*} V^*$ and
$Q_{\hat\pi_T^d}^*=R_{\hat\pi_T^d}+\alpha P_{\hat\pi_T^d}V^*$, where
\begin{align*}
&P_{\hat\pi_T^d}=\sum_{a\in {\cal A}}{\hat\pi_{a,T}^d(s)e_s^T e_s
P_a},\quad R_{\hat\pi_T^d}=\sum_{a\in {\cal
A}}{\hat\pi_{a,T}^d(s)e_s^T e_s R_a},
\end{align*}
~\eqref{eq:explain-17} follows from
$R_{\hat\pi_T^d}=V^{\hat\pi_T^d}-\alpha P_{\hat\pi_T^d} V^{\hat\pi_T^d}$,
and~\eqref{eq:explain-18} follows
from~\cref{corollary:uniqueness-of-dual-LP-DP}. This completes the
proof.
\end{proof}

With the above result, one can derive a convergence result of the
policy constructed from the dual variables. The proof follows that
of~\citet[Theorem~4]{chen2016stochastic}. In particular,~\cref{lemma:duality-gap-representation} leads to
\begin{align}
L_I(\hat x_T,y^*)-L_I(x^*,\hat y_T) &= \sum_{a'\in {\cal A}}
{\hat\lambda_{a',T}^T}(I-\alpha P_{\hat \pi_T^d})(V^*-V^{\hat\pi_T^d})\nonumber\\
&\ge\eta^T(I-\alpha P_{\hat\pi_T^d})(V^*-V^{\hat\pi_T^d})\label{eq:explain-19}\\
&\ge\min_{s\in {\cal S}}(\eta^T e_s)\|(I-\alpha
P_{\hat\pi_T^d})(V^*-V^{\hat\pi_T^d})\|_\infty\nonumber\\
&\ge \min_{s \in {\cal S}}(\eta^T
e_s)(\|V^*-V^{\hat\pi_T^d}\|_\infty-\|\alpha P_{\hat\pi_T^d}(V^*-V^{\hat\pi_T^d})\|_\infty)\nonumber\\
&\ge\min_{s\in {\cal S}}(\eta^T e_s)(\|V^*-V^{\hat\pi_T^d}\|_\infty
-\|\alpha P_{\hat\pi_T^d}\|_\infty \|V^*-V^{\hat\pi_T^d}\|_\infty)\nonumber\\
&=\min_{s\in {\cal S}}(\eta^Te_s)(1-\alpha)\|V^*-V^{\hat\pi_T^d}\|_\infty,\label{eq:explain-20}
\end{align}
where~\eqref{eq:explain-19} is due to the constraint set $\Xi$ and~\eqref{eq:explain-20} is due to $\|P_{\hat\pi_T^d}\| =1$ and $\|\alpha P_{\hat\pi_T^d}\| = \alpha\|P_{\hat\pi_T^d}\|= \alpha$.
Combining the last inequality
with~\cref{proposition:duality-gap-convergence}, we have that for
any $\varepsilon>0$ and $\delta\in (0,1/e)$, if
$T\ge\max\{\Psi_1,\Psi_2\}$, then with probability at least
$1-\delta$, $\|V^*-V^{\hat\pi_T^d}\|_\infty\le\frac{1}{\min_{s \in
{\cal S}}(\eta^T e_s)(1-\alpha)}\varepsilon$ holds. Replacing
$\varepsilon$ with $\min_{s \in {\cal S}}(\eta^T
e_s)(1-\alpha)\varepsilon$ in the iteration lower bound~\ref{proposition:duality-gap-convergence}, and after
simplifications, we have the desired conclusion.

\section{Simulations}\label{section:simulations}

\subsection{Simple discounted MDP}

We consider the discounted MDP $({\cal S},{\cal A},{\cal P},{\cal R},\alpha)$ with ${\cal S} = \{1,2\}$, ${\cal A} = \{1,2\}$, ${\cal R}:= \{ \hat r_{ss'a}  \in [0,\sigma ],a
\in {\cal A},s,s' \in {\cal S}\}$ with $\hat r_{11} =3$, $\hat r_{12} = 1$, $\hat r_{21} = 2$, $\hat r_{22} = 1$, $\alpha = 0.9$, $\sigma = 3$, and
\begin{align*}
&P_1 = \begin{bmatrix}
0.2&0.8\\
0.3&0.7
\end{bmatrix},\quad P_2 = \begin{bmatrix}
0.5&0.5\\
0.7&0.3
\end{bmatrix}.
\end{align*}
In addition, consider the initial state distribution $v_0 = \begin{bmatrix}
0.4&0.6\end{bmatrix}^T$, the behavior policy $\theta_1 = \begin{bmatrix}
0.2&0.8\end{bmatrix}^T,\theta_2=\begin{bmatrix}0.7&0.3\end{bmatrix}^T$ and set $\eta=0.1\begin{bmatrix}1&1\end{bmatrix}^T$. The transition probability matrix under $\theta$ is $P_\theta = \begin{bmatrix}
0.44&0.56\\
0.42&0.58
\end{bmatrix}$ and the corresponding stationary state distribution is $\lim_{t\to \infty}v_0 P_\theta^t=v_\infty=\begin{bmatrix}
0.4286&0.5714\end{bmatrix}^T$. Then, the matrix $M_\infty$ corresponding to~\eqref{eq:M-matrix-def} is computed as
\begin{align*}
&M_{\infty,1}=\begin{bmatrix}
\theta_1(1) v_\infty(1)&0\\
0&\theta_2(1) v_\infty(2)
\end{bmatrix}= \begin{bmatrix}
0.0857&0\\
0&0.4000
\end{bmatrix},\\
&M_{\infty,2}=\begin{bmatrix}
\theta_1(2)v_\infty(1)&0\\
0&\theta_2(2)v_\infty(2)
\end{bmatrix}=\begin{bmatrix}
0.3429&0\\
0&0.1714
\end{bmatrix},\\
&M_\infty=\begin{bmatrix}
0.0857&0&0&0\\
0&0.4&0&0\\
0&0&0.3429&0\\
0&0&0&0.1714
\end{bmatrix}.
\end{align*}
One can also numerically compute $\zeta = 0.0856$, which is in general not available in real-world applications because it requires the knowledge on the state-action distributions for all $k\geq 0$. Solving the primal LP~\eqref{eq:LP-form-Q-learning} yields the primal optimal solution
\begin{align*}
&Q_1^*= \begin{bmatrix}
20.0690\\
18.1931
\end{bmatrix},\quad Q_2^*=\begin{bmatrix}
19.4414\\
18.6897
\end{bmatrix},\quad V^*=\begin{bmatrix}
20.0690\\
18.6897
\end{bmatrix},
\end{align*}
while by solving the dual LP~\eqref{eq:dual-LP-form-Q-learning}, the dual optimal solution is obtained as
\begin{align*}
&\lambda_1^*=\begin{bmatrix}
0.9379\\
0
\end{bmatrix},\quad \lambda_2^*= \begin{bmatrix}
0\\
1.0621
\end{bmatrix},\\
&\mu_1^*= \begin{bmatrix}
10.9425\\
0
\end{bmatrix},\quad \mu_2^*=\begin{bmatrix}
0\\
6.1954\end{bmatrix}.
\end{align*}
The corresponding optimal policy constructed from the dual solution is
\begin{align*}
&\pi_1^*:=\begin{bmatrix}
\frac{\lambda_1^*(1)}{\lambda_1^*(1)+\lambda_2^*(1)}&\frac{\lambda_2^*(1)}{\lambda_1^*(1)+\lambda_2^*(1)}
\end{bmatrix}^T=\begin{bmatrix}
1&0
\end{bmatrix}^T\in\Delta_2,\\
&\pi_2^*:=\begin{bmatrix}
\frac{\lambda_1^*(2)}{\lambda _1^*(2) + \lambda _2^*(2)}&\frac{\lambda_2^*(2)}{\lambda_1^*(2)+\lambda_2^*(2)}
\end{bmatrix}^T= \begin{bmatrix}
0&1
\end{bmatrix}^T\in \Delta_2.
\end{align*}
We run~\cref{algo:primal-dual-method} with $T = 10^5$, $\gamma_k=\gamma_0/\sqrt{k+1},k \ge 0$, $\eta = \frac{\sigma}{|{\cal S}|} \bf{1}_{|{\cal S}|}$, and~\cref{fig:fig1} depicts the evolutions of the Q-function error, defined as $\sum_{a\in {\cal A}}{{\|Q_a^*-{\hat Q}_{a,T}\|}_\infty}$, of~\cref{algo:primal-dual-method}, for different $\gamma_0 \in \{1,2,\ldots,4\}$.
\begin{figure}[h!]
\centering\includegraphics[width=15cm,height=8cm]{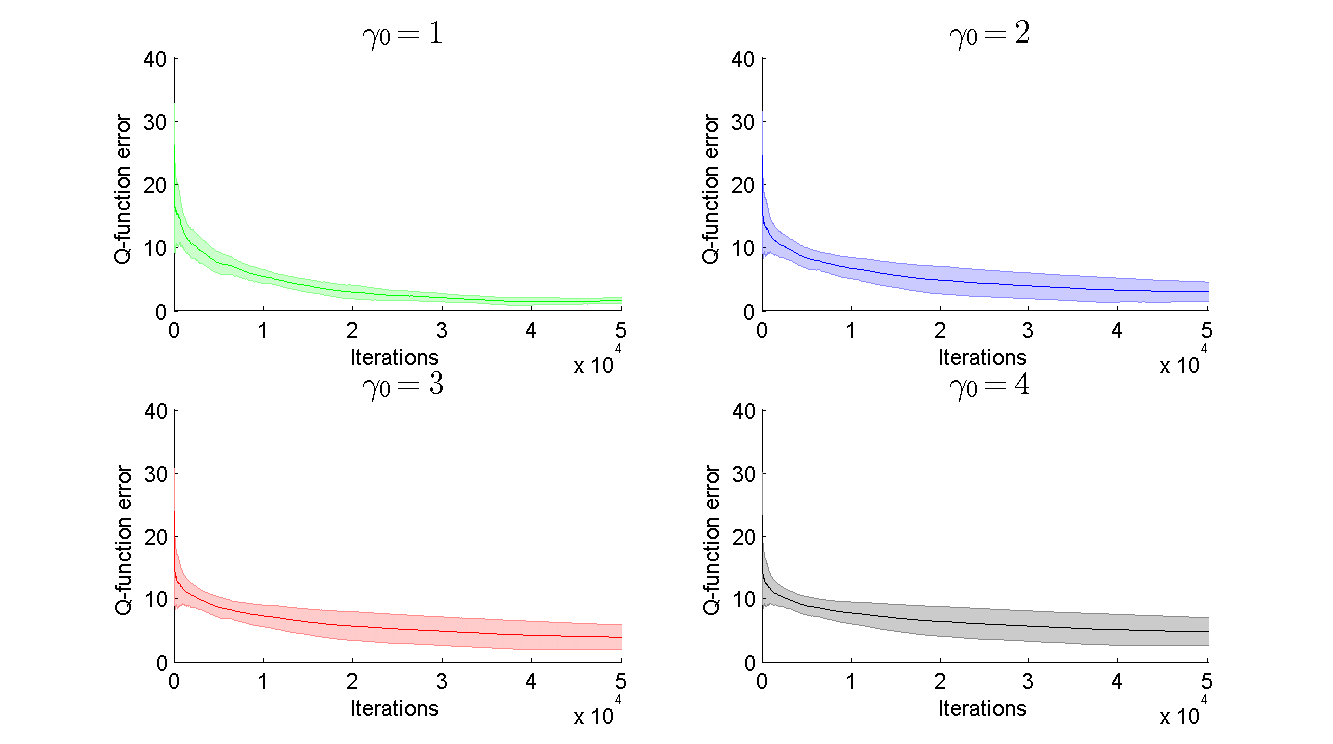}
\caption{Evolution of the Q-function error, $\sum_{a\in {\cal A}}{{\|Q_a^*-{\hat Q}_{a,T}\|}_\infty}$ for the SPD Q-learning.}\label{fig:fig1}
\end{figure}

The evolutions of the dual policy error, $\sum_{s\in {\cal S}} {\|\pi_s^*-\hat\pi_{s,T}^d\|_2}$, obtained using~\cref{algo:primal-dual-method} for different $\gamma_0\in\{1,2,\ldots,4\}$, are given in~\cref{fig:fig2} (left-hand side).
\begin{figure}[h!]
\centering\includegraphics[width=15cm,height=11cm]{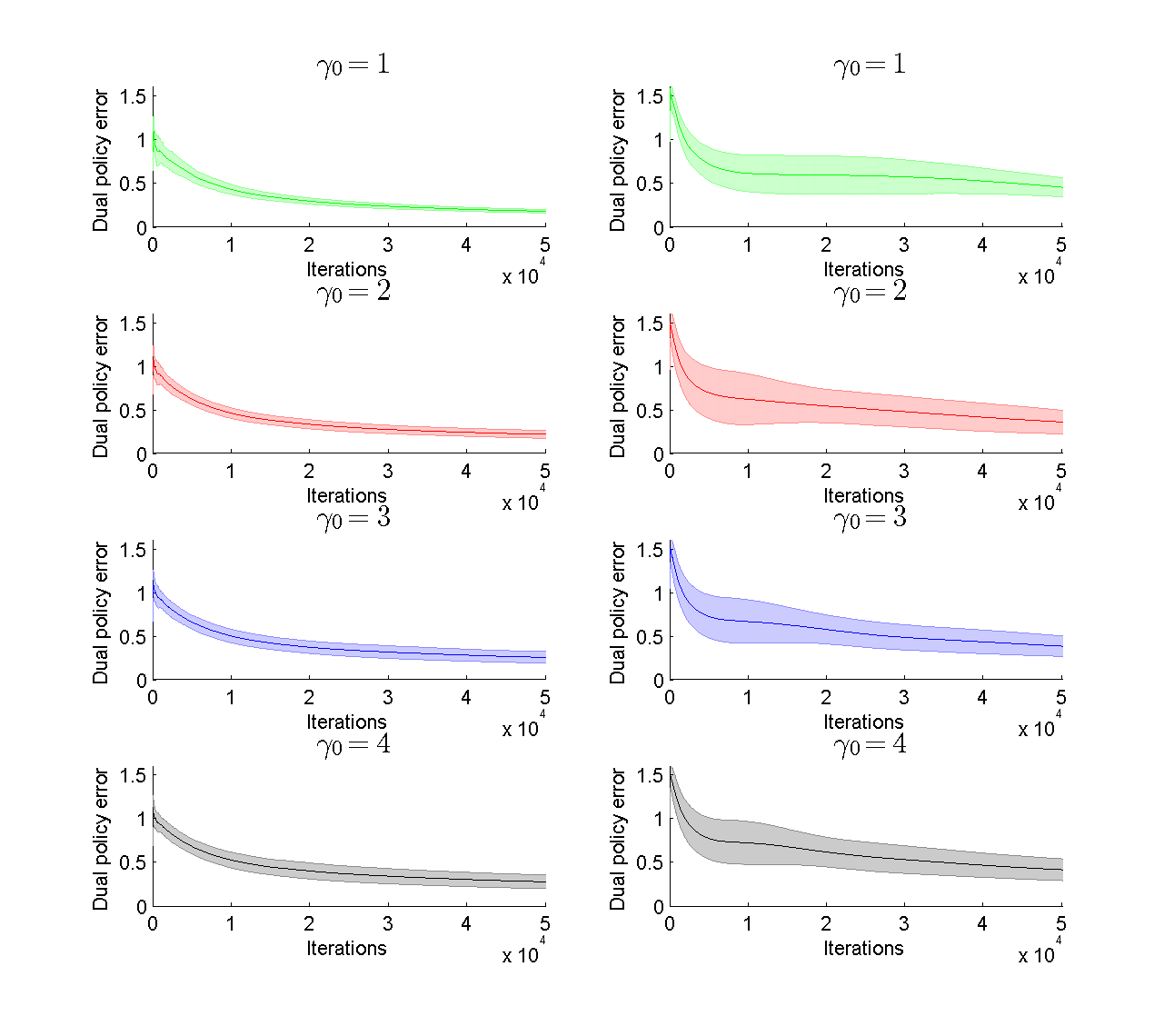}
\caption{Evolution of the dual policy error, $\sum_{s\in {\cal S}} {\|\pi_s^*-\hat\pi_{s,T}^d\|_\infty}$, from the SPD Q-learning and the dual policy error, $\sum_{s\in {\cal S}} {\|\pi_s^*-\tilde\pi_{s,T}\|_\infty}$, from the SPD-RL in~\citet[Algorithm~1]{chen2016stochastic} with the importance sampling.}\label{fig:fig2}
\end{figure}
In the figure, the result is compared with the error, $\sum_{s\in {\cal S}} {\|\pi_s^*-\bar \pi_{s,T}\|_\infty}$ (right-hand side), of the stochastic policy $\bar \pi_{s,T}$ obtained by using the dual solutions of a modified~\citet[Algorithm~1]{chen2016stochastic} with $\eta = \sigma {\bf 1}_{|{\cal S}|}/|{\cal S}|$. Note that in the modified algorithm, the dual solutions of~\citet[Algorithm~1]{chen2016stochastic} are multiplied by $\hat M_T$, which estimates the true $M_T$ by sample averages so as to find the true optimal dual variables, and all algorithms for the comparison employ the step-size rule, $\gamma_k=\gamma_0/\sqrt{k+1},k \ge 0$. \cref{fig:fig2} implies that both~\cref{algo:primal-dual-method} and modified~\citet[Algorithm~1]{chen2016stochastic} demonstrate similar convergence results in terms of the dual policy errors. Moreover, it shows that the dual policy from~\cref{algo:primal-dual-method} outperforms that from the modified SPD algorithm,~\citet[Algorithm~1]{chen2016stochastic}. This is reasonable as the latter suffers from additional estimation errors.

\cref{fig:fig3} shows the primal policy error (right-hand side), $\sum_{s\in {\cal S}} {\|\pi_s^*-\tilde\pi_{s,T}^p\|_\infty}$, where
\begin{align*}
&\tilde\pi_{s,T}^p(s):=\begin{cases}
   \begin{bmatrix}
   1 & 0\\
\end{bmatrix}^T \quad {\rm if}\quad {\rm argmax}_{a\in {\cal A}} \hat Q_{a,T}(s)=1\\
   \begin{bmatrix}
   0 & 1\\
\end{bmatrix}^T \quad {\rm if}\quad {\rm argmax}_{a\in {\cal A}} \hat Q_{a,T}(s)=2  \\
\end{cases},
\end{align*}
and the right-hand side figures are the policy error corresponding to the standard Q-learning. As one can see that, SPD Q-learning algorithm performs worse than the standard Q-learning on this simple task, which is more or less expected since Q-learning is a very powerful algorithm in practice. What's interesting here is that, when comparing the dual policy error in~\cref{fig:fig2} to the primal policy error in~\cref{fig:fig3}, it is clear that the primal policy of SPD Q-learning converges much faster than the dual policy. This demonstrates another potential advantage of  the proposed algorithm over the existing primal-dual algorithm.
\begin{figure}[h!]
\centering\includegraphics[width=15cm,height=11cm]{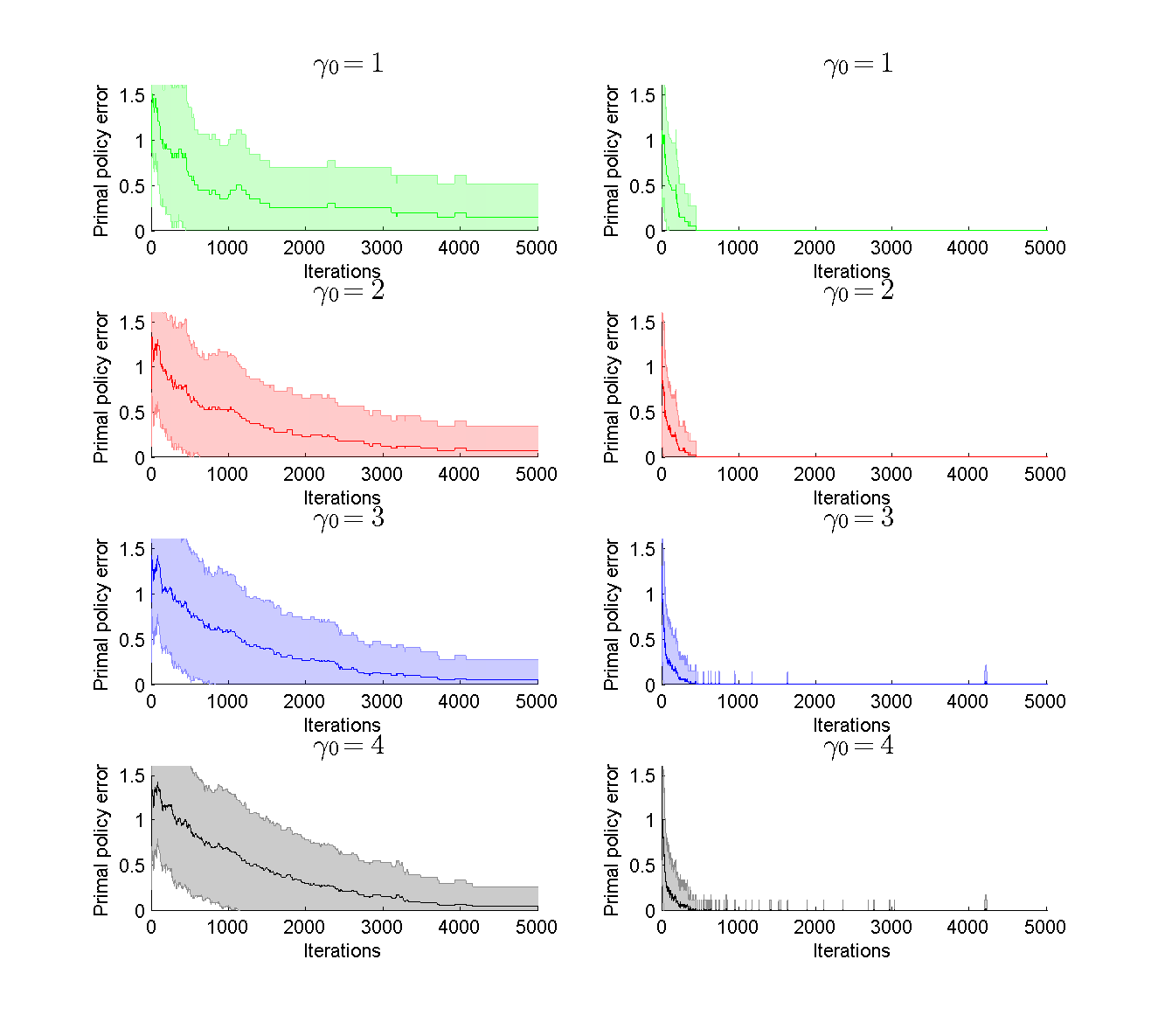}
\caption{Evolution of the primal policy error, $\sum_{s\in {\cal S}}{\|\pi_s^*-\tilde\pi_{s,T}^p\|_\infty}$ (left-hand side), from the SPD Q-learning and the error of the standard Q-learning (right-hand side).}\label{fig:fig3}
\end{figure}

\subsection{$2 \times 2$ grid world}

In this example, we consider a $2 \times 2$ grid world, which simulates a path-planning problem for a mobile robot in an environment. The goal of the RL agent is to navigate from the starting point (left-bottom corner) to the goal (right-top corner), using four actions ${\cal A}=\{{\rm up,down,left,right}\}$. The behavior policy is defined as a stochastic policy which uniformly chooses one among the four actions. If the action leads the agent to escape the square boundary, then the location of the agent does not change. The reward is uniformly distributed in $[0,0.2]$ except for the reward at the goal state which is uniformly distributed in $[1,1.2]$.  We run~\cref{algo:primal-dual-method} with $T = 5000$, $\gamma_k=2/\sqrt{k+10000},k \ge 0$, $\eta = \frac{\sigma}{|{\cal S}|} \bf{1}_{|{\cal S}|}$, and $\alpha = 0.9$.

\cref{fig:fig4} illustrates the evolutions of the average reward corresponding to the primal policy of the SPD Q-learning (blue line) and the average reward of the standard Q-learning (green line). At each iteration step, the average rewards are obtained by the sample average of the rewards under the primal policy at the iteration step over eight time steps. The results show that the average reward of the SPD Q-learning converges to that of the standard Q-learning.
\begin{figure}[h!]
\centering\includegraphics[width=12cm,height=7cm]{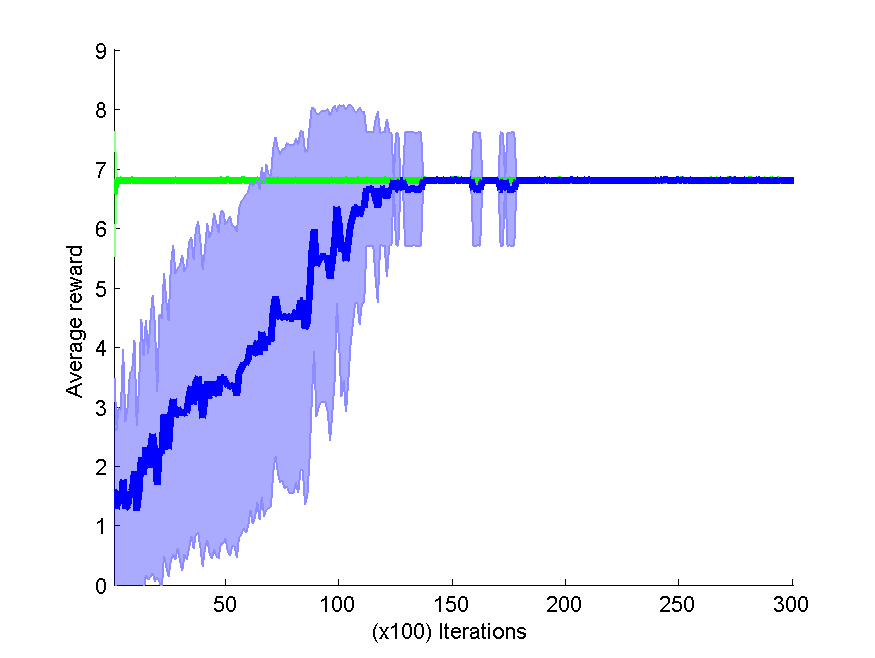}
\caption{Evolution of the average reward corresponding to the primal policy of the SPD Q-learn
ing (blue line) and the average reward of the standard Q-learning (green line).}\label{fig:fig4}
\end{figure}

\section{Conclusion}\label{section:conclusion}
In this paper, we introduce a new SPD-RL algorithm, where real-world observations under arbitrary behavior policies are used for finding a near-optimal policy. We prove the convergence with its sample complexity analysis. Promising future research directions are summarized as follows:
\begin{enumerate}
\item Safe RL: There exist scenarios where the safety of the
RL agent is critical, where one should take into account the safety during and after the learning while maximizing the long-term reward. In this case, the dual LP~\eqref{eq:dual-LP-form-Q-learning} is useful in that the optimal dual variables represent the state-action distribution under the optimal policy. By imposing constraints on the dual variable in~\eqref{eq:dual-LP-form-Q-learning}, we can shape the distribution by including prior knowledge of the task to design safer policies which avoid certain risks.

\item Distributed RL: In distributed RL~\citep{lee2018primal}, each agent receives local reward through a local processing, while communicating over sparse and random networks to learn the global value function corresponding to the aggregate of local rewards. The distributed learning can be formulated as a distributed optimization, and the frameworks in this paper can be applied for policy design problems.

\item Function approximation: The proposed SPD Q-learning framework can be easily combined with  (linear or nonlinear) function approximations to handle large-scale or continuous state-action spaces. In this approach, both primal and dual variables need to be approximated by parameterized function classes, where important questions arise: if the algorithm converges, then how the resulting value function and policies can be interpreted? Can we derive meaningful optimality error bounds? It remains interesting to explore the theoretical and empirical convergence behaviors of the algorithm under such extensions.

\end{enumerate}

\appendix
\section*{Appendix A. Proof of~\cref{lemma:K1-K2}}

In this appendix, we prove~\cref{lemma:K1-K2} from~\cref{section:convergence}:

\noindent
{\bf Lemma} {\it Assume that there exists a real number $\zeta>0$ such that it is
less than or equal to any diagonal element of $M$. Then, we have
\begin{align*}
&\|\nabla_x L_{M_k}(x_k,y_k)+\varepsilon_k\|_2 \le
\frac{\sqrt{13}|{\cal S}||{\cal A}| \|\eta\|_1
}{\zeta(1-\alpha)} =:K_1,\\
&\|\nabla_y L_{M_k}(x_k,y_k)+\xi_k\|_2 \le \frac{\sqrt{13} |{\cal
S}||{\cal A}|\sigma}{1-\alpha}=:K_2.
\end{align*}}
\begin{proof}
The first inequality follows by the chains of inequalities
\begin{align}
&\|\nabla_x L_{M_k}(x_k,y_k)+\varepsilon_k\|_2\nonumber\\
&= \left\| \begin{bmatrix}
   (e_{\hat s_k}  \otimes e_{\hat a_k})e_{\hat s_k}^T |{\cal S}||{\cal A}|\lambda_{\hat a_k}^{(k)}  - (e_{s_k}\otimes e_{a_k})e_{s_k}^T \mu_{a_k}^{(k)}\\
   e_{\hat s_k} e_{\hat s_k}^T |{\cal S}|\eta-e_{\hat s_k} e_{\hat s_k}^T |{\cal S}||{\cal A}|\lambda_{\hat a_k}^{(k)}+\alpha e_{s_{k+1}} e_{s_k}^T \mu_{a_k}^{(k)}\\
\end{bmatrix} \right\|_2\nonumber\\
&= (\| (e_{\hat s_k}\otimes e_{\hat a_k})e_{\hat s_k}^T
|{\cal S}||{\cal A}|\lambda_{\hat a_k}^{(k)}  - (e_{s_k}\otimes
e_{a_k})e_{s_k}^T \mu_{a_k}^{(k)} \|_2^2\nonumber\\
&+\|e_{\hat s_k}
e_{\hat s_k}^T |{\cal S}|\eta  - e_{\hat s_k} e_{\hat s_k}^T
|{\cal S}||{\cal A}|\lambda_{\hat a_k}^{(k)}+\alpha e_{s_{k+1}}
e_{s_k}^T \mu_{a_k}^{(k)}\|_2^2)^{1/2}\nonumber\\
&\le (2|{\cal S}|^2 |{\cal A}|^2 \| \lambda^{(k)} \|_\infty
^2  + 2\|\mu^{(k)}\|_\infty^2 + 3|{\cal S}|^2 \| \eta \|_\infty^2
+ 3|{\cal S}|^2 |{\cal A}|^2 \| \lambda^{(k)} \|_\infty^2+\alpha^2
3\|\mu ^{(k)}\|_\infty^2)^{1/2}\nonumber\\
&\le \left(\frac{5|{\cal S}|^2 |{\cal A}|^2 \| \eta
\|_1^2}{(1-\alpha)^2} + \frac{5\| \eta \|_1^2}{\zeta^2
(1-\alpha)^2} + 3|{\cal S}|^2 \| \eta \|_\infty^2 \right)^{1/2}\nonumber\\
&\le \frac{\sqrt{13} |{\cal S}||{\cal A}|
\|\eta\|_1}{\zeta(1-\alpha)}, \label{eq:eq10}
\end{align}
where the first inequality follows from the vector inequality $\|
v_1  + v_2  +  \cdots  + v_n \|_2^2  \le n(\| v_1 \|_2^2+\| v_2
\|_2^2 + \cdots + \|v_n\|_2^2)$ for any vectors $v_1,\ldots,v_n$
and $\| e_{\hat s_k} e_{\hat s_k}^T \lambda_{\hat a_k}^{(k)}\|_2^2
\le \| \lambda^{(k)} \|_\infty ^2,\|e_{s_{k + 1}} e_{s_k}^T
\mu_{a_k}^{(k)} \|_2^2  \le \| \mu ^{(k)} \|_\infty^2$, the second
inequality follows from~\cref{lemma:bound-of-solutions} and the
last inequality is obtained after simplifications. For the second
result, we have
\begin{align*}
&\| \nabla_y L_{M_k}(x_k,y_k)+\xi_k \|_2\\
&= \left\| \begin{bmatrix}
   |{\cal S}||{\cal A}|[(e_{\hat s_k}  \otimes e_{\hat a_k})e_{\hat s_k }^T Q_{\hat a_k}^{(k)}  - (e_{\hat s_k} \otimes e_{\hat a_k})e_{\hat s_k}^T V^{(k)}]\\
   \alpha e_{s_k} e_{s_{k + 1}}^T V^{(k)} + e_{s_k} \hat r(s_k,a_k,s_{k + 1}) - e_{s_k} e_{s_k}^T Q_{a_k}^{(k)}\\
\end{bmatrix} \right\|_2\\
&= (\| |{\cal S}||{\cal A}|[(e_{\hat s_k}\otimes e_{\hat a_k
})e_{\hat s_k}^T Q_{\hat a_k}^{(k)}-(e_{\hat s_k}\otimes
e_{\hat a_k})e_{\hat s_k}^T V^{(k)}]\|_2^2\\
&+ \|\alpha e_{s_k}e_{s_{k+1}}^T V^{(k)} + e_{s_k} \hat
r(s_k,a_k,s_{k+1})-e_{s_k} e_{s_k}^T Q_{a_k}^{(k)}\|_2^2)^{1/2}\\
&\le (2|{\cal S}|^2|{\cal A}|^2 \|Q^{(k)}\|_\infty^2+2
\|V^{(k)}\|_\infty ^2 + 3\alpha^2
\|V^{(k)}\|_\infty^2+3\sigma^2+3\|Q^{(k)}\|_\infty^2)^{1/2}\\
&\le \left(2|{\cal S}|^2 |{\cal A}|^2
\frac{\sigma^2}{(1-\alpha)^2} + 2|{\cal S}|^2|{\cal A}|^2
\frac{\sigma^2}{(1-\alpha)^2} + 3\alpha^2
\frac{\sigma^2}{(1-\alpha)^2} + 3\sigma^2+3\frac{\sigma^2}{(1-\alpha)^2}\right)^{1/2}\\
&\le \sqrt{13} \frac{|{\cal S}||{\cal A}|\sigma}{1-\alpha},
\end{align*}
where the first inequality follows from the vector inequality
$\left\| {v_1  + v_2  +  \cdots  + v_n } \right\|_2^2 \le
n(\left\| {v_1 } \right\|_2^2  + \left\| {v_2 } \right\|_2^2 +
\cdots  + \left\| {v_n } \right\|_2^2 )$ for any vectors $v_1 ,
\ldots ,v_n $ and $\| e_{s_k } e_{s_{k + 1} }^T V^{(k)}\|_2^2 \le
\| V^{(k)} \|_\infty^2,\| e_{s_k} e_{s_k}^T Q_{a_k}^{(k)} \|_2^2
\le \|Q^{(k)}\|_\infty^2$, the second inequality follows
from~\cref{lemma:bound-of-solutions}, and the last inequality
follows after simplifications.
\end{proof}

\vskip 0.2in

\section*{Appendix B. Proof of~\cref{lemma:Phi-bounds}}

In this appendix, we prove~\cref{lemma:Phi-bounds} from~\cref{section:convergence}:

\noindent
{\bf Lemma} {\it Assume that there exists a real number $\zeta>0$ such that it is
less than or equal to any diagonal element of $M$. Then, we have
\begin{align*}
&\Phi_1(x^*)\le\frac{2\sigma^2 |{\cal S}|}{(1-\alpha)^2}\frac{1}{\gamma_{T-1}},\\
&\Phi_2(y^*)\le \frac{1}{\gamma_{T-1}}\frac{4\| \eta \|_1^2}{\zeta^4
(1-\alpha)^2}+\frac{2 \|\eta
\|_1^2}{\zeta^3(1-\alpha)^2}\sum_{k=1}^{T-1}{\frac{\beta_{k-1}}{\gamma_{k-1}}}.
\end{align*}
}
\begin{proof}
First, $\Phi_1(x^*)$ is bounded by using the chains of inequalities
\begin{align}
\Phi_1(x^*)&= \sum_{k=0}^{T-1}{\frac{1}{2\gamma_k}({\cal E}_k^{(1)}(x^*)-{\cal E}_{k+1}^{(1)}(x^*))}\nonumber\\
&\le\frac{1}{2}\left(\frac{1}{\gamma_0}{\cal
E}_0^{(1)}(x^*)+\sum_{k=0}^{T-2}{\left(
\frac{1}{\gamma_{k+1}}-\frac{1}{\gamma_k} \right){\cal
E}_{k+1}^{(1)}(x^*)} \right)\nonumber\\
&= \frac{1}{2}\left(\frac{1}{\gamma_0} \| x_0-x^*
\|_2^2+\sum_{k=0}^{T-2} {\left(
\frac{1}{\gamma_{k+1}}-\frac{1}{\gamma_k}
\right)\|x_{k+1}-x^*\|_2^2} \right)\nonumber\\
&\le \frac{1}{2}\frac{1}{\gamma_0}(\|x_0\|_2^2+\|x^*\|_2^2+2\|x_0\|_2 \|x^*\|_2)\nonumber\\
&+ \frac{1}{2}\sum_{k=0}^{T-2}{\left(
\frac{1}{\gamma_{k+1}}-\frac{1}{\gamma_k}\right)(\|x_{k+1}\|_2^2+\|x^*\|_2^2+ 2\|x_{k+1}\|_2\|x^*\|_2)}\label{eq:explain-2}\\
&\le\frac{1}{2}\frac{4\sigma^2|{\cal S}|}{(1-\alpha
)^2}\left(\frac{1}{\gamma_0} + \sum_{k=0}^{T-2}{\left(
\frac{1}{\gamma_{k+1}}-\frac{1}{\gamma_k}\right)}
\right)\label{eq:explain-3}\\
&= \frac{2\sigma^2 |{\cal S}|}{(1-\alpha)^2}\frac{1}{\gamma_{T-1}},\nonumber
\end{align}
where~\eqref{eq:explain-2} follows from the relation $\|a-b\|_2^2=\|a\|_2^2+\|b\|_2^2-2a^T b$ and Cauchy-Schwarz inequality and~\eqref{eq:explain-3} is due to $x_k ,x^*\in {\cal
X}$. Similarly, we have
\begin{align}
\Phi_2(y^*) &=\frac{1}{2}\sum_{k=0}^{T-1}{\frac{1}{\gamma_k}({\cal
E}_k^{(2)} (\bar M_k^{-1} y^*)-{\cal
E}_{k+1}^{(2)} (\bar M_k^{-1} y^*))}\nonumber\\
&\le \frac{1}{2}\left( \frac{1}{\gamma_0}{\cal E}_0^{(2)}
(\bar M_0^{-1} y^*) + \sum_{k=1}^{T-1} {\left(\frac{1}{\gamma_k}{\cal
E}_k^{(2)} (\bar M_k^{-1} y^*)-\frac{1}{\gamma_{k-1}}{\cal E}_k^{(2)}
(\bar M_{k-1}^{-1} y^*)\right)} \right)\nonumber\\
&= \frac{1}{2}\frac{1}{\gamma_0}{\cal E}_0^{(2)} (\bar M_0^{-1}
y^*)+\frac{1}{2}\sum_{k=1}^{T-1}{\left(
\frac{1}{\gamma_k}-\frac{1}{\gamma_{k-1}}
\right){\cal E}_k^{(2)} (\bar M_k^{-1} y^*)}\nonumber\\
&+ \frac{1}{2}\sum_{k=1}^{T-1} {\frac{1}{\gamma_{k-1}}({\cal
E}_k^{(2)} (\bar M_k^{-1} y^*)-{\cal E}_k^{(2)}(\bar M_{k-1}^{-1}
y^*))},\label{eq:eq8}
\end{align}
where the last equality is obtained by rearranging terms. For any
$k\geq 0$, ${\cal E}_k^{(2)}(\bar M_k^{-1} y^*)$ is bounded as
\begin{align*}
{\cal E}_k^{(2)}(\bar M_k^{-1} y^*)&=\|y_k - \bar M_k^{-1}y^*\|_2^2\\
&= \|y_k\|_2^2 + \|\bar M_k^{-1}y^*\|_2^2 - 2y_k^T \bar M_k^{-1}y^*\\
&\le \|y_k\|_2^2 + \|\bar M_k^{-1}y^* \|_2^2 + 2\|y_k\|_2 \|\bar M_k^{-1}y^*\|_2\\
&\le \|y_k\|_2^2 + \|\bar M_k^{-1}\|_2^2 \|y^*\|_2^2 + 2\|y_k\|_2\|\bar M_k^{-1}\|_2 \|y^*\|_2\\
&\le \|y_k\|_2^2 +\zeta^{-2}\|y^*\|_2^2 + 2\|y_k\|_2 \zeta^{-1} \|y^*\|_2\\
&\le \frac{\|\eta \|_1^2}{\zeta^2(1-\alpha)^2} + \zeta^{-2}
\frac{\|\eta\|_1^2}{\zeta^2(1-\alpha)^2} + 2\zeta^{-1}
\frac{\|\eta\|_1^2}{\zeta^2(1-\alpha)^2}\\
&\le \frac{4\| \eta \|_1^2}{\zeta^4(1-\alpha)^2},
\end{align*}
where the first inequality is due to the Cauchy-Schwarz inequality, the fourth inequality is due
to~\cref{lemma:bound-of-solutions}, and the last inequality follows using $\zeta <1$. Upon substituting the above inequality into~\eqref{eq:eq8}, we obtain
\begin{align}
&\Phi_2(y^*)\le\frac{1}{\gamma_{T-1}}\frac{4\| \eta
\|_1^2}{\zeta^4(1-\alpha)^2} + \frac{1}{2}\sum_{k=1}^{T-1}
{\frac{1}{\gamma_{k-1}}({\cal E}_k^{(2)}(\bar M_k^{-1} y^*)-{\cal
E}_k^{(2)}(\bar M_{k-1}^{-1} y^*))}.\label{eq:eq9}
\end{align}
The second term in~\eqref{eq:eq9} is written as
\begin{align}
&\frac{1}{2}\sum_{k=1}^{T-1}{\frac{1}{\gamma_{k-1}}({\cal
E}_k^{(2)} (\bar M_k^{-1} y^*)-{\cal E}_k^{(2)}(\bar M_{k-1}^{-1} y^*))}\nonumber\\
&=\frac{1}{2}\sum_{k=1}^{T-1}{\frac{1}{\gamma_{k-1}}(\|y_k-\bar M_k^{-1}y^*\|_2^2
-\|y_k-\bar M_{k-1}^{-1} y^*\|_2^2)}\nonumber\\
&=\frac{1}{2}\sum_{k=1}^{T-1}{\frac{1}{\gamma_{k-1}}(\|\bar M_k^{-1}y^*\|_2^2
+\|y_k\|_2^2 -2y_k^T \bar M_k^{-1}y^*-
\|\bar M_{k-1}^{-1}y^*\|_2^2-\|y_k\|_2^2 + 2y_k^T \bar M_{k-1}^{-1}y^*)}\label{eq:explain-4}\\
&=\frac{1}{2}\sum_{k=1}^{T-1}{\frac{1}{\gamma_{k-1}}((y^*)^T(\bar M_k^{-1}
\bar M_k^{-1}-\bar M_{k-1}^{-1} \bar M_{k-1}^{-1})y^*+2y_k^T (\bar M_{k-1}^{-1}-\bar M_k^{-1})y^*)}\nonumber\\
&=\frac{1}{2}\sum_{k=1}^{T-1}{\frac{1}{\gamma_{k-1}}(y^*(\bar M_k^{-1}
-\bar M_{k-1}^{-1})(\bar M_k^{-1}+\bar M_{k-1}^{-1})y^*+2y_k^T(\bar M_{k-1}^{-1}-\bar M_k^{-1})y^*)}\label{eq:explain-5}\\
&\le\frac{1}{2}\sum_{k=1}^{T-1}{\frac{1}{\gamma_{k-1}}(\|y^*\|_2^2
\|\bar M_{k-1}^{-1}-\bar M_k^{-1}\|_2 \|\bar M_{k-1}^{-1}+ \bar M_k^{-1}\|_2+
2\|y_k\|_2 \|y^*\|_2 \|\bar M_{k-1}^{-1}- \bar M_k^{-1}\|_2)}\nonumber\\
&\leq\frac{1}{2}\frac{\|\eta\|_1^2}{\zeta^2(1-\alpha)^2}\sum_{k=1}^{T-1}
{\frac{1}{\gamma_{k-1}}(\|\bar M_{k-1}^{-1}-\bar M_k^{-1}\|_2 \|\bar M_{k-1}^{-1}
+\bar M_k^{-1}\|_2+2\|\bar M_{k-1}^{-1}-\bar M_k^{-1}\|_2)}\label{eq:explain-6}\\
&\le\frac{1}{2}\frac{\|\eta\|_1^2}{\zeta^2(1-\alpha)^2}\sum_{k=1}^{T-1}
{\frac{1}{\gamma_{k-1}}(\|\bar M_{k-1}^{-1} - \bar M_k^{-1}\|_2
\frac{2}{\zeta}+2\|\bar M_{k-1}^{-1}-\bar M_k^{-1}\|_2)}\label{eq:explain-7}\\
&= \frac{\|\eta\|_1^2}{\zeta^2(1-\alpha)^2}\left(1 +
\frac{1}{\zeta}\right)\sum_{k=1}^{T-1}{\frac{\|\bar M_{k-1}^{-1}-\bar M_k^{-1}\|}{\gamma_{k-1}}}\nonumber\\
&\le \frac{2\|\eta\|_1^2}{\zeta^3(1-\alpha)^2}\sum_{k=1}^{T-1}
{\frac{\beta_{k-1}}{\gamma_{k-1}}},\nonumber
\end{align}
where~\eqref{eq:explain-4} follows from the relation $\|a-b\|_2^2
=\|a\|_2^2+\|b\|_2^2-2a^T b$ for any vectors
$a,b$,~\eqref{eq:explain-5} follows from
$(\bar M_k^{-1}-\bar M_{k-1}^{-1})(\bar M_k^{-1}+\bar M_{k-1}^{-1})=\bar M_k^{-1}
\bar M_k^{-1}-\bar M_{k-1}^{-1} \bar M_{k-1}^{-1}$,~\eqref{eq:explain-6} is due
to $y_k,y_T^*\in {\cal Y}$,~\eqref{eq:explain-7} follows from the
triangle inequality and $\|\bar M_k^{-1}\|_2\le \zeta^{-1}$, and the
last inequality comes from~\cref{assump:distribution-convergence}.
Combining the last inequality with~\eqref{eq:eq9} yields the
second conclusion. This completes the proof.
\end{proof}

\vskip 0.2in

\section*{Appendix C. Proof of~\cref{lemma:bound1}}

In this appendix, we prove~\cref{lemma:bound1} from~\cref{section:convergence}:

\noindent
{\bf Lemma} {\it We have $\Delta {\cal G}_{T+1}={\cal G}_{T+1}-{\cal
G}_T=\frac{1}{2\gamma_T}({\cal E}_{T+1}-{\mathbb E}[{\cal
E}_{T+1}|{\cal F}_T])\le b$ with probability one, where
\begin{align*}
&b=\frac{13\gamma_0 \|\eta\|_1+
13\gamma_0\zeta^2\sigma^2+16\sqrt{26}\sigma\zeta
\|\eta\|_1}{2}\frac{|{\cal S}|^2 |{\cal
A}|^2}{\zeta^2(1-\alpha)^2}.
\end{align*}}
\begin{proof}
Note that
\begin{align*}
&{\cal G}_{T+1}-{\cal G}_T=\frac{1}{2\gamma_T}({\cal
E}_{T+1}^{(1)}(x_T^*)- {\mathbb E}[{\cal E}_{T+1}^{(1)}(x_T^*)|{\cal
F}_T])+\frac{1}{2\gamma_T}({\cal E}_{T+1}^{(2)} (\bar M_{T+1}^{-1}y_T^*)
-{\mathbb E}[{\cal E}_{T+1}^{(2)}(\bar M_{T+1}^{-1} y_T^*)|{\cal F}_T]),
\end{align*}
and we first derive a bound on $\frac{1}{2\gamma_T}({\cal
E}_{T+1}^{(1)}(x_T^*)-{\mathbb E}[{\cal E}_{T+1}^{(1)}(x_T^*)|{\cal
F}_T])$. Using the definition of ${\cal E}_{k+1}^{(1)}(x_T^*)$
yields
\begin{align}
&\frac{1}{2\gamma_T}({\cal E}_{T+1}^{(1)}(x_T^*)-{\mathbb E}[{\cal
E}_{T+1}^{(1)}(x_T^*)|{\cal F}_T])=\frac{1}{2\gamma_T} \|
\Pi_{\cal X}(x_T-\gamma_T\nabla_x L_{G_k}(x_T,y_T)-\gamma_T\varepsilon_T)-x_T^* \|_2^2\nonumber\\
&-\frac{1}{2\gamma_T}{\mathbb E}[\| \Pi_{\cal X}(x_T-\gamma_T
\nabla_x L_{G_T}(x_T,y_T)-\gamma_T\varepsilon_T)-x_T^* \|_2^2 |{\cal
F}_T].\label{eq:eq6}
\end{align}
Noting
\begin{align*}
&\|\Pi_{\cal X}(x_T-\gamma_T\nabla_x L_{G_T}(x_T,y_T)-\gamma_T
\varepsilon_T)-x_T^*\|_2^2\\
& =\|\Pi_{\cal X}(x_T-\gamma_T\nabla_x
L_T(x_T,y_T)-\gamma_T \varepsilon_T)-x_T\|_2^2+\|x_T-x_T^*\|_2^2\\
&-2(x_T-x_T^*)^T(\Pi_{\cal X}(x_T-\gamma_T\nabla_x
L_T(x_T,y_T)-\gamma_T \varepsilon_T)-x_T)
\end{align*}
it follows from~\eqref{eq:eq6} that
\begin{align}
&\frac{1}{2\gamma_T}({\cal E}_{T+1}^{(1)}(x_T^*)- {\mathbb E}[{\cal
E}_{T+1}^{(1)}(x_T^*)|{\cal F}_T ])\nonumber\\
&=\frac{1}{2\gamma_T} \|\Pi_{\cal X}(x_T-\gamma_T\nabla_x
L_{G_T}(x_T,y_T)-\gamma_T\varepsilon_T)-x_T\|_2^2+\frac{1}{2\gamma_T}\|x_T-
x_T^*\|_2^2\nonumber\\
&-\frac{1}{\gamma_T}(x_T-x_T^*)^T(\Pi_{\cal X}(x_T-\gamma_T\nabla_x
L_{G_T}(x_T,y_T)-\gamma_T\varepsilon_T)-x_T)\nonumber\\
&-\frac{1}{2\gamma_T} {\mathbb E}[\| \Pi_{\cal X}(x_T-\gamma_T
\nabla_x L_{G_T}(x_T,y_T)-\gamma_T\varepsilon_T)-x_T\|_2^2|{\cal
F}_T]-\frac{1}{2\gamma_T} \|x_T-x_T^*\|_2^2\nonumber\\
&+\frac{1}{\gamma_T}{\mathbb E}[(\Pi_{\cal X}(x_T-\gamma_T\nabla_x
L_{G_T}(x_T,y_T)-\gamma_T
\varepsilon_T)-x_T)^T(x_T-x_T^*)|{\cal F}_T]\nonumber\\
&\le \frac{1}{2\gamma_T}\|\Pi_{\cal X}(x_T-\gamma_T\nabla_x
L_T(x_T,y_T)-\gamma_T\varepsilon_T)-x_T\|_2^2\nonumber\\
&+\frac{1}{\gamma_T}\|x_T-x_T^*\|_2\|\Pi_{\cal X}(x_T-\gamma_T
\nabla_x L_T(x_T,y_T)-\gamma_T \varepsilon_T)-x_T\|_2\nonumber\\
&+ \frac{1}{\gamma_T} {\mathbb E}[\|x_T-x_T^*\|_2 \|\Pi_{\cal
X}(x_T-\gamma_T\nabla_x L_T(x_T,y_T)-\gamma_T\varepsilon_k)-x_T\|_2 |{\cal F}_T]\label{eq:explain-1}\\
&\le\frac{\gamma_T}{2} \|\nabla_x L_T(x_T,y_T)+\varepsilon_T\|_2^2
+\|x_T-x_T^*\|_2 \|\nabla_x L_T(x_T,y_T)+\varepsilon_T\|_2\nonumber\\
&+ \|x_T-x_T^*\|_2 {\mathbb E}[\|\nabla_x L_T(x_T,y_T)+\varepsilon_T\|_2 |{\cal F}_T]\label{eq:explain-9}\\
&\le\frac{1}{2}(\gamma_T K_1^2+4K_1 \|x_T-x_T^*\|_2)\nonumber\\
&\le\frac{1}{2}\left(\gamma_0\left( \frac{\sqrt{13}|{\cal
S}||{\cal A}| \|\eta\|_1}{\zeta(1-\alpha)} \right)^2 + 4\left(
\frac{\sqrt{13}|{\cal S}||{\cal A}| \|\eta\|_1}{\zeta(1-\alpha)}
\right)\sqrt{|{\cal S}||{\cal A}| + |{\cal S}|}
\|x_T-x_T^*\|_\infty \right)\nonumber\\
&\le\frac{13\gamma_0+8\sqrt{26}\sigma\zeta}{2}\frac{|{\cal S}|^2
|{\cal A}|^2 \|\eta\|_1}{\zeta^2(1-\alpha)^2},\nonumber
\end{align}
where~\eqref{eq:explain-1} follows from the Cauchy-Schwarz
inequality,~\eqref{eq:explain-9} follows from the nonexpansive map
property of the projection $\|\Pi_{\cal X}(a)-\Pi_{\cal X}(b)\|_2
\le \|a-b\|_2$, and the last inequality is obtained after
simplifications. Using similar lines, one obtains
\begin{align*}
&\frac{1}{2\gamma_T}({\cal E}_{T+1}^{(2)}(\bar M_T^{-1}y_T^*) - {\mathbb
E}[{\cal E}_{T+1}^{(2)}(\bar M_T^{-1}y_T^*)|{\cal F}_T])\le
\frac{13\gamma_0\zeta^2\sigma + 8\sqrt{26}\zeta
\|\eta\|_1}{2}\frac{\sigma |{\cal S}|^2 |{\cal
A}|^2}{\zeta^2(1-\alpha)^2},
\end{align*}
and combining the two inequalities completes the proof.
\end{proof}

\vskip 0.2in

\section*{Appendix D. Proof of~\cref{lemma:bound2}}

In this appendix, we prove~\cref{lemma:bound2} from~\cref{section:convergence}:

\noindent
{\bf Lemma} {\it $\frac{1}{T} \langle {\cal G} \rangle_T\le a $ holds with
probability one, where
\begin{align*}
&a=\frac{1}{4}\frac{(\gamma_0(13 \|\eta\|_1 +
4\sqrt{26}\sigma\zeta)\|\eta\|_1+13\gamma_0\zeta^2\sigma^2+4\sqrt{26}\sigma
\|\eta\|_1)^2 |{\cal S}|^4 |{\cal A}|^4}{\zeta^4(1-\alpha)^4}.
\end{align*}}
\begin{proof}
Using ${\mathbb E}[{\cal E}_k |{\cal F}_k]={\cal E}_k$, we have
\begin{align}
{\mathbb E}[|{\cal G}_{k+1}-{\cal G}_k|^2 |{\cal
F}_k]&=\frac{1}{4\gamma_k^2} {\mathbb E}[|{\cal E}_{k+1}- {\mathbb
E}[{\cal E}_{k+1}|{\cal F}_k]|^2 |{\cal F}_k]\nonumber\\
&= \frac{1}{4\gamma_k^2} {\mathbb E}[|{\cal E}_{k+1}-{\cal E}_k- {\mathbb E}[{\cal E}_{k+1}-{\cal E}_k |{\cal F}_k]|^2|{\cal F}_k]\nonumber\\
&\le \frac{1}{4\gamma_k^2} {\mathbb E}[{\mathbb E}[|{\cal
E}_{k+1}-{\cal E}_k|^2 |{\cal F}_k]|{\cal F}_k]\nonumber\\
&= \frac{1}{4\gamma_k^2}{\mathbb E}[|\underbrace{{\cal
E}_{k+1}^{(1)}(x_T^*)+{\cal E}_{k+1}^{(2)}(\bar M_{k+1}^{-1} y_T^*)-{\cal
E}_k^{(1)}(x_T^*)-{\cal E}_k^{(2)}(\bar M_{k+1}^{-1} y_T^*)}_{=:\Upsilon_1}|^2
|{\cal F}_k],\label{eq:eq7}
\end{align}
where the inequality follows from the fact that the variance of a
random variable is bounded by its second moment. For
bounding~\eqref{eq:eq7}, note that $\Phi_1$ is written as
\begin{align*}
&\Upsilon_1= \|x_{k+1}-x_T^*\|_2^2 - \|x_k-x_T^*\|_2^2 + \|y_{k+1}- \bar M_{k+1}^{-1} y_T^*\|_2^2- \|y_k-\bar M_{k+1}^{-1} y_T^*\|_2^2
\end{align*}
and $|\Upsilon_1|\le |\|x_{k+1}-x_T^*\|^2-\|x_k-x_T^*\|^2| + |\|y_{k+1}-\bar M_{k+1}^{-1}y_T^*\|^2 - \|y_k-\bar M_{k+1}^{-1}y_T^*\|^2|$.

Here, the first two terms have the bound
\begin{align}
&|\|x_{k+1}-x_T^*\|^2 - \|x_k - x_T^*\|^2|\nonumber\\
&= |\|\Pi_{\cal X}(x_k -\gamma_k\nabla_x L_k(x_k,y_k)-\gamma_k
\varepsilon_k)-x_k\|^2\nonumber\\
& - 2(\Pi_{\cal X}(x_k-\gamma_k \nabla_x
L_k(x_k,y_k)-\gamma_k\varepsilon_k)-x_k)^T(x_k-x_T^*)|\label{eq:explain-10}\\
&\le \|\Pi_{\cal X} (x_k-\gamma_k\nabla_x L_k(x_k,y_k)-\gamma_k
\varepsilon_k)-x_k\|_2^2\nonumber\\
& + 2\|\Pi_{\cal X}(x_k-\gamma_k\nabla_x
L_k(x_k,y_k)-\gamma_k\varepsilon_k)-x_k \|_2 \|x_k-x_T^*\|_2\label{eq:explain-11}\\
&\le\gamma_k^2 \|\nabla_x
L_k(x_k,y_k)+\varepsilon_k\|_2^2+2\gamma_k \|\nabla_x L_k(x_k,y_k)+\varepsilon_k\|_2 \|x_k-x_T^*\|_2\label{eq:explain-13}\\
&\le\gamma_k^2 K_1^2 +2\gamma_k K_1\sqrt{|{\cal S}||{\cal
A}|+|{\cal S}|} \|x_k-x_T^*\|_\infty\label{eq:explain-12}\\
&\le\gamma_k^2 \left( \frac{\sqrt{13}|{\cal S}||{\cal A}|
\|\eta\|_1}{\zeta(1-\alpha)} \right)^2 +2\gamma_k \frac{\sqrt{13}
|{\cal S}||{\cal A}| \|\eta\|_1}{\zeta(1-\alpha)}\sqrt{2|{\cal S}||{\cal A}|}\frac{2\sigma}{1-\alpha}\nonumber\\
&\le\gamma_k \frac{\gamma_0 (13\|\eta\|_1+
4\sqrt{26}\sigma\zeta)|{\cal S}|^2 |{\cal A}|^2 \|\eta\|_1}{\zeta^2(1-\alpha)^2}, \nonumber
\end{align}
where~\eqref{eq:explain-10} follows from the relation $\|a-b\|_2^2
=\|a\|_2^2+\| b\|_2^2 - 2a^T b$ for any vectors
$a,b$,~\eqref{eq:explain-11} follows from the Cauchy-Schwarz
inequality,~\eqref{eq:explain-13} is due to the nonexpansive map
property of the projection $\|\Pi_{\cal X}(a)-\Pi_{\cal X}(b)\|_2
\le \|a-b\|_2$,~\eqref{eq:explain-12} comes
from~\cref{lemma:K1-K2} and the inequality $\|a\|_2\le \sqrt n
\|a\|_\infty$ for any $a\in {\mathbb R}^n$, and the last
inequality follows from algebraic simplifications. Similarly, the
second two terms in $\Phi_1$ are bounded as
\begin{align*}
&|\|y_{k+1}-\bar M_{k+1}^{-1} y_T^*\|^2-\|y_k-\bar M_{k+1}^{-1}y_T^*\|^2|\\
&=|\|\Pi_{\cal Y}(y_k-\gamma_k\nabla_y
L_{M_k}(x_k,y_k)-\gamma_k\xi_k)-y_k\|^2\\
&-2(\Pi_{\cal Y}(y_k-\gamma_k\nabla_y L_{M_k}(x_k,y_k)-\gamma_k\xi_k)-y_k)(y_k-\bar M_{k+1}^{-1}y_T^*)|\\
&\le \|\Pi_{\cal Y}(y_k-\gamma_k\nabla_y L_{M_k}(x_k,y_k)-\gamma_k \xi_k)-y_k\|_2^2\\
& + 2\|\Pi_{\cal Y}(y_k-\gamma_k\nabla_y L_{M_k}(x_k,y_k)-\gamma_k
\xi_k)-y_k\|_2 \|y_k-\bar M_{k+1}^{-1} y_T^*\|_2\\
&\le\gamma_k^2 \|\nabla_y L_{M_k}(x_k,y_k)+\xi_k\|_2^2+2\gamma_k
\|\nabla_y L_{M_k}(x_k,y_k)+\xi_k\|_2 \|y_k-\bar M_{k+1}^{-1}y_T^*\|_2\\
&\le\gamma_k^2 K_2^2 + 2\gamma_k K_2\sqrt{2|{\cal S}||{\cal A}|}
(\|y_k\|_\infty+\|\bar M_{k+1}^{-1}\|_\infty \|y_T^*\|_\infty)\\
&\le\gamma_k \frac{(13\gamma_0\zeta^2\sigma^2+4\sqrt{26}\sigma\|
\eta\|_1)|{\cal S}|^2|{\cal A}|^2}{\zeta^2(1-\alpha)^2}.
\end{align*}
Combining the last two results leads to
\begin{align*}
&|\Upsilon_1|\le\gamma_k \frac{(\gamma_0(13
\|\eta\|_1+4\sqrt{26}\sigma\zeta)
\|\eta\|_1+13\gamma_0\zeta^2\sigma^2+4\sqrt{26}\sigma
\|\eta\|_1)|{\cal S}|^2|{\cal A}|^2}{\zeta^2(1-\alpha)^2},
\end{align*}
and plugging the bound on $|\Upsilon_1|$ into~\eqref{eq:eq7} and after
simplifications, we obtain
\begin{align*}
&{\mathbb E}[|{\cal G}_{k+1}-{\cal G}_k |^2 |{\cal F}_k]\le
\frac{1}{4\gamma_k^2} {\mathbb E}[\Upsilon_1^2|{\cal F}_k]\\
&\le\frac{1}{4}\frac{(\gamma_0(13\|\eta\|_1+4\sqrt{26}\sigma\zeta)\|\eta\|_1+13\gamma_0\zeta^2
\sigma^2+4\sqrt{26}\sigma \|\eta\|_1)^2 |{\cal S}|^4 |{\cal
A}|^4}{\zeta^4(1-\alpha)^4},
\end{align*}
which is the desired conclusion.
\end{proof}

\vskip 0.2in

\section*{Appendix~E. Proofs of~\cref{ex:ex3}}

\noindent
{\bf \cref{proposition:ex3}} {\it (a) There exists some real number $c >0$ such that $\beta_k \leq c |\lambda_2|^k$. }
\begin{proof}
Since
$u_1,u_2,\ldots,u_{|{\cal S}|}$ span ${\mathbb R}^{|{\cal S}|}$,
one can write $v_0=c_1 u_1 +\cdots+c_{|{\cal S}|}u_{|{\cal S}|}$
with $c_1,c_2,\ldots,c_{|{\cal S}|}\in {\mathbb R}$ so that
\begin{align*}
&v_k=(P_\theta^T)^k v_0= U^{-T}\Sigma^k U^T v_0 =U^{-T}\Sigma^k
U^T(c_1u_1+\cdots+c_{|{\cal S}|}u_{|{\cal S}|})= c_1
\lambda_1^k u_1+\cdots+c_{|{\cal S}|}\lambda_{|{\cal S}|}^k u_{|{\cal S}|}\\
&=c_1 u_1 + c_2 \lambda_2^k u_2+\cdots+ c_{|{\cal
S}|}\lambda_{|{\cal S}|}^k u_{|{\cal S}|}
\end{align*}
and $v_k(s)=c_1 e_s^T u_1 +c_2\lambda_2^k e_s^T
u_2+\cdots+c_{|{\cal S}|}\lambda_{|{\cal S}|}^k e_s^T u_{|{\cal
S}|}$. Then, we get
\begin{align}
&\beta_k=\zeta^{-2} \max_{s \in {\cal S}}(v_{k+1}(s)-v_k(s))\nonumber\\
&=\zeta^{-2} \max_{s \in {\cal S}}((c_1 e_s^T u_1 + \cdots +
c_{|{\cal S}|}\lambda_{|{\cal S}|}^{k+1} e_s^T u_{|{\cal
S}|})-(c_1 e_s^T u_1  + \cdots + c_{|{\cal S}|}\lambda_{|{\cal
S}|}^k e_s^T u_{|{\cal S}|}))\nonumber\\
&=\zeta^{-2} \max_{s\in {\cal S}}(c_2(\lambda_2-1)\lambda_2^k
e_s^T u_2+\cdots+ c_{|{\cal S}|}(\lambda_{|{\cal S}|}-1)\lambda_{|{\cal S}|}^k e_s^T u_{|{\cal S}|})\nonumber\\
&\le\zeta^{-2}\max_{s\in {\cal S}}(|c_2(\lambda_2-1)e_s^T u_2
||\lambda_2|^k+\cdots+ |c_{|{\cal S}|}(\lambda_{|{\cal S}|}-1)e_s^T u_{|{\cal S}|}||\lambda_{|{\cal S}|}|^k)\nonumber\\
&\le\zeta^{-2} \max_{s\in {\cal S}}(|c_2(\lambda_2-1)e_s^T
u_2|+\cdots+|c_{|{\cal S}|}(\lambda_{|{\cal S}|}-1)e_s^T u_{|{\cal
S}|}|)|\lambda_2|^k.\label{eq:eq14}
\end{align}
Therefore, setting $c = \zeta^{-2} \max_{s\in {\cal S}}(|c_2(\lambda_2-1)e_s^T
u_2|+\cdots+|c_{|{\cal S}|}(\lambda_{|{\cal S}|}-1)e_s^T u_{|{\cal
S}|}|)$ gives the desired conclusion.
\end{proof}

\noindent
{\bf \cref{proposition:ex3}} {\it (b) There exists some real number $\kappa>0$ such that $|\lambda_2|^k\le
\kappa/(k+1)$ for all $k\geq 0$.}
\begin{proof}
We first calculate a
nonnegative integer $k^*$ such that $|\lambda_2|^k \le
1/(k+1),\forall k \ge k^*$. Noting
\begin{align}
&|\lambda_2|^k\le \frac{1}{k+1} \Leftrightarrow \left(
\frac{1}{|\lambda_2|}\right)^k\ge k+1 \Leftrightarrow k \ge
\log_{1/|\lambda_2|}(k+1)=\frac{\ln (k+1)}{-\ln |\lambda_2|}
\Leftrightarrow \exp(-\ln |\lambda_2|k) \ge k + 1\label{eq:eq13}
\end{align}
and using the Taylor expansion, a sufficient condition
for~\eqref{eq:eq13} is
\begin{align*}
&\exp(-\ln |\lambda_2|k)=\sum_{n=0}^\infty{\frac{(-\ln
|\lambda_2|k)^n}{n!}}\ge 1 -\ln |\lambda_2|k + \frac{(\ln
|\lambda_2|)^2 k^2}{2}\ge k+1.
\end{align*}
Solving the last inequality leads to the conclusion that with
$k^*=\max(0,\left\lceil \frac{2\ln |\lambda_2|+2}{(\ln
|\lambda_2|)^2} \right\rceil)$, we have $|\lambda_2 |^k\le
\frac{1}{k+1},\forall k \ge k^*$. Equivalently, one has
\begin{align*}
&|\lambda_2|^{k+k^*}\le \frac{1}{k+k^*+1},\forall k \ge 0.
\end{align*}
Noting $|\lambda_2|^k |\lambda_2|^{k^*}\le \frac{1}{k+k^*+1}\le
\frac{1}{k+1},\forall k \ge 0$, one concludes $|\lambda_2|^k\le
\frac{|\lambda_2|^{-k^*}}{k+1},\forall k\ge 0$. Therefore, setting $\kappa = |\lambda_2|^{-k^*}$ concludes the proof.
\end{proof}

\bibliography{reference}

\begin{thebibliography}{36}
\providecommand{\natexlab}[1]{#1}
\providecommand{\url}[1]{\texttt{#1}}
\expandafter\ifx\csname urlstyle\endcsname\relax
  \providecommand{\doi}[1]{doi: #1}\else
  \providecommand{\doi}{doi: \begingroup \urlstyle{rm}\Url}\fi

\bibitem[Baird(1995)]{baird1995residual}
Leemon Baird.
\newblock Residual algorithms: reinforcement learning with function
  approximation.
\newblock In \emph{Machine Learning Proceedings}, pages 30--37. 1995.

\bibitem[Bercu et~al.(2015)Bercu, Delyon, and Rio]{bercu2015concentration}
Bernard Bercu, Bernard Delyon, and Emmanuel Rio.
\newblock \emph{Concentration inequalities for sums and martingales}.
\newblock Springer, 2015.

\bibitem[Bertsekas and Tsitsiklis(1996)]{bertsekas1996neuro}
Dimitri~P. Bertsekas and John~N. Tsitsiklis.
\newblock \emph{Neuro-dynamic programming}.
\newblock Athena Scientific Belmont, MA, 1996.

\bibitem[Bertsekas et~al.(2003)Bertsekas, Nedi{\'c}, and
  Ozdaglar]{bertsekas2003convex}
Dimitri~P. Bertsekas, Angelia Nedi{\'c}, and Asuman~E. Ozdaglar.
\newblock Convex analysis and optimization.
\newblock 2003.

\bibitem[Boyd and Vandenberghe(2004)]{Boyd2004}
Stephen Boyd and Lieven Vandenberghe.
\newblock \emph{Convex optimization}.
\newblock Cambridge University Press, 2004.

\bibitem[Chen and Wang(2016)]{chen2016stochastic}
Yichen Chen and Mengdi Wang.
\newblock Stochastic primal-dual methods and sample complexity of reinforcement
  learning.
\newblock \emph{arXiv preprint arXiv:1612.02516}, 2016.

\bibitem[Chen et~al.(2017)Chen, Everett, Liu, and How]{chen2017socially}
Yu~Fan Chen, Michael Everett, Miao Liu, and Jonathan~P. How.
\newblock Socially aware motion planning with deep reinforcement learning.
\newblock In \emph{IEEE/RSJ International Conference on Intelligent Robots and
  Systems (IROS)}, pages 1343--1350, 2017.

\bibitem[Dai et~al.(2017)Dai, He, Pan, Boots, and Song]{dai2017learning}
Bo~Dai, Niao He, Yunpeng Pan, Byron Boots, and Le~Song.
\newblock Learning from conditional distributions via dual embeddings.
\newblock In \emph{Artificial Intelligence and Statistics}, pages 1458--1467,
  2017.

\bibitem[Dai et~al.(2018{\natexlab{a}})Dai, Shaw, He, Li, and
  Song]{dai2018boosting}
Bo~Dai, Albert Shaw, Niao He, Lihong Li, and Le~Song.
\newblock Boosting the actor with dual critic.
\newblock In \emph{International Conference on Learning Representations},
  2018{\natexlab{a}}.

\bibitem[Dai et~al.(2018{\natexlab{b}})Dai, Shaw, Li, Xiao, He, Liu, Chen, and
  Song]{dai2018sbeed}
Bo~Dai, Albert Shaw, Lihong Li, Lin Xiao, Niao He, Zhen Liu, Jianshu Chen, and
  Le~Song.
\newblock Sbeed: convergent reinforcement learning with nonlinear function
  approximation.
\newblock In \emph{International Conference on Machine Learning}, pages
  1133--1142, 2018{\natexlab{b}}.

\bibitem[Fan et~al.(2012)Fan, Grama, and Liu]{fan2012hoeffding}
Xiequan Fan, Ion Grama, and Quansheng Liu.
\newblock Hoeffding's inequality for supermartingales.
\newblock \emph{Stochastic Processes and their Applications}, 122\penalty0
  (10):\penalty0 3545--3559, 2012.

\bibitem[Freedman(1975)]{freedman1975tail}
David~A. Freedman.
\newblock On tail probabilities for martingales.
\newblock \emph{The Annals of Probability}, pages 100--118, 1975.

\bibitem[Garc{\i}a and Fern{\'a}ndez(2015)]{garcia2015comprehensive}
Javier Garc{\i}a and Fernando Fern{\'a}ndez.
\newblock A comprehensive survey on safe reinforcement learning.
\newblock \emph{Journal of Machine Learning Research}, 16\penalty0
  (1):\penalty0 1437--1480, 2015.

\bibitem[Gentle(2007)]{gentle2007matrix}
James~E. Gentle.
\newblock Matrix algebra.
\newblock \emph{Springer Texts in Statistics}, 2007.

\bibitem[Kar et~al.(2013)Kar, Moura, and Poor]{kar2013cal}
Soummya Kar, Jos{\'e}~MF Moura, and H.~Vincent Poor.
\newblock {QD}-learning: a collaborative distributed strategy for multi-agent
  reinforcement learning through consensus $+$ innovations.
\newblock \emph{IEEE Transactions on Signal Processing}, 61\penalty0
  (7):\penalty0 1848--1862, 2013.

\bibitem[Lee et~al.(2018)Lee, Yoon, and Hovakimyan]{lee2018primal}
Donghwan Lee, Hyungjin Yoon, and Naira Hovakimyan.
\newblock Primal-dual algorithm for distributed reinforcement learning:
  distributed {GTD}.
\newblock \emph{57th IEEE Conference on Decision and Control (accepted)}, 2018.

\bibitem[Longstaff and Schwartz(2001)]{longstaff2001valuing}
Francis~A. Longstaff and Eduardo~S. Schwartz.
\newblock Valuing american options by simulation: {A} simple least-squares
  approach.
\newblock \emph{The review of financial studies}, 14\penalty0 (1):\penalty0
  113--147, 2001.

\bibitem[Macua et~al.(2015)Macua, Chen, Zazo, and Sayed]{macua2015distributed}
Sergio~Valcarcel Macua, Jianshu Chen, Santiago Zazo, and Ali~H. Sayed.
\newblock Distributed policy evaluation under multiple behavior strategies.
\newblock \emph{IEEE Transactions on Automatic Control}, 60\penalty0
  (5):\penalty0 1260--1274, 2015.

\bibitem[Mahadevan and Liu(2012)]{mahadevan2012sparse}
Sridhar Mahadevan and Bo~Liu.
\newblock Sparse {Q}-learning with mirror descent.
\newblock \emph{arXiv preprint arXiv:1210.4893}, 2012.

\bibitem[Mahadevan et~al.(2014)Mahadevan, Liu, Thomas, Dabney, Giguere, Jacek,
  Gemp, and Liu]{mahadevan2014proximal}
Sridhar Mahadevan, Bo~Liu, Philip Thomas, Will Dabney, Steve Giguere, Nicholas
  Jacek, Ian Gemp, and Ji~Liu.
\newblock Proximal reinforcement learning: {A} new theory of sequential
  decision making in primal-dual spaces.
\newblock \emph{arXiv preprint arXiv:1405.6757}, 2014.

\bibitem[Mnih et~al.(2015)Mnih, Kavukcuoglu, Silver, Rusu, Veness, Bellemare,
  Graves, Riedmiller, Fidjeland, Ostrovski, et~al.]{mnih2015human}
Volodymyr Mnih, Koray Kavukcuoglu, David Silver, Andrei~A Rusu, Joel Veness,
  Marc~G Bellemare, Alex Graves, Martin Riedmiller, Andreas~K. Fidjeland, Georg
  Ostrovski, et~al.
\newblock Human-level control through deep reinforcement learning.
\newblock \emph{Nature}, 518\penalty0 (7540):\penalty0 529, 2015.

\bibitem[Nedi{\'c} and Ozdaglar(2009)]{nedic2009subgradient}
Angelia Nedi{\'c} and Asuman Ozdaglar.
\newblock Subgradient methods for saddle-point problems.
\newblock \emph{Journal of optimization theory and applications}, 142\penalty0
  (1):\penalty0 205--228, 2009.

\bibitem[Ng and Russell(2000)]{ng2000algorithms}
Andrew~Y. Ng and Stuart~J. Russell.
\newblock Algorithms for inverse reinforcement learning.
\newblock In \emph{International Conference on Machine Learning}, pages
  663--670, 2000.

\bibitem[Precup et~al.(2001)Precup, Sutton, and Dasgupta]{precup2001off}
Doina Precup, Richard~S. Sutton, and Sanjoy Dasgupta.
\newblock Off-policy temporal-difference learning with function approximation.
\newblock In \emph{International Conference on Machine Learning}, pages
  417--424, 2001.

\bibitem[Puterman(2014)]{puterman2014markov}
Martin~L. Puterman.
\newblock \emph{Markov decision processes: {D}iscrete stochastic dynamic
  programming}.
\newblock John Wiley \& Sons, 2014.

\bibitem[Resnick(2013)]{resnick2013adventures}
Sidney~I. Resnick.
\newblock \emph{Adventures in stochastic processes}.
\newblock Springer Science \& Business Media, 2013.

\bibitem[Rummery and Niranjan(1994)]{rummery1994line}
Gavin~A. Rummery and Mahesan Niranjan.
\newblock \emph{On-line {Q}-learning using connectionist systems}, volume~37.
\newblock University of Cambridge, Department of Engineering Cambridge,
  England, 1994.

\bibitem[Strehl et~al.(2009)Strehl, Li, and Littman]{strehl2009reinforcement}
Alexander~L. Strehl, Lihong Li, and Michael~L. Littman.
\newblock Reinforcement learning in finite {MDP}s: {PAC} analysis.
\newblock \emph{Journal of Machine Learning Research}, 10:\penalty0 2413--2444,
  2009.

\bibitem[Sutton(1988)]{sutton1988learning}
Richard~S. Sutton.
\newblock Learning to predict by the methods of temporal differences.
\newblock \emph{Machine learning}, 3\penalty0 (1):\penalty0 9--44, 1988.

\bibitem[Sutton and Barto(1998)]{sutton1998reinforcement}
Richard~S. Sutton and Andrew~G. Barto.
\newblock \emph{Reinforcement learning: {A}n introduction}.
\newblock MIT Press, 1998.

\bibitem[Sutton et~al.(2009{\natexlab{a}})Sutton, Maei, Precup, Bhatnagar,
  Silver, Szepesv{\'a}ri, and Wiewiora]{sutton2009fast}
Richard~S. Sutton, Hamid~R. Maei, Doina Precup, Shalabh Bhatnagar, David
  Silver, Csaba Szepesv{\'a}ri, and Eric Wiewiora.
\newblock Fast gradient-descent methods for temporal-difference learning with
  linear function approximation.
\newblock In \emph{Proceedings of the 26th Annual International Conference on
  Machine Learning}, pages 993--1000, 2009{\natexlab{a}}.

\bibitem[Sutton et~al.(2009{\natexlab{b}})Sutton, Maei, and
  Szepesv{\'a}ri]{sutton2009convergent}
Richard~S. Sutton, Hamid~R. Maei, and Csaba Szepesv{\'a}ri.
\newblock A convergent {$O(n)$} temporal-difference algorithm for off-policy
  learning with linear function approximation.
\newblock In \emph{Advances in neural information processing systems}, pages
  1609--1616, 2009{\natexlab{b}}.

\bibitem[Tesauro and Kephart(2002)]{tesauro2002pricing}
Gerald Tesauro and Jeffrey~O. Kephart.
\newblock Pricing in agent economies using multi-agent {Q}-learning.
\newblock \emph{Autonomous Agents and Multi-Agent Systems}, 5\penalty0
  (3):\penalty0 289--304, 2002.

\bibitem[Wang and Chen(2016)]{wang2016online}
Mengdi Wang and Yichen Chen.
\newblock An online primal-dual method for discounted markov decision
  processes.
\newblock In \emph{55th IEEE Conference on Decision and Control (CDC)}, pages
  4516--4521, 2016.

\bibitem[Watkins and Dayan(1992)]{watkins1992q}
Christopher J. C.~H. Watkins and Peter Dayan.
\newblock Q-learning.
\newblock \emph{Machine learning}, 8\penalty0 (3-4):\penalty0 279--292, 1992.

\bibitem[Zhang et~al.(2018)Zhang, Yang, Liu, Zhang, and
  Ba{\c{s}}ar]{zhang2018fully}
Kaiqing Zhang, Zhuoran Yang, Han Liu, Tong Zhang, and Tamer Ba{\c{s}}ar.
\newblock Fully decentralized multi-agent reinforcement learning with networked
  agents.
\newblock \emph{arXiv preprint arXiv:1802.08757}, 2018.

\end{thebibliography}

\end{document}